\documentclass[12pt]{amsart}
\usepackage[english]{babel}

\usepackage{amsmath,amsthm,amssymb,amsfonts,pifont,mathtools}
\usepackage{verbatim}
\usepackage{times}
\usepackage{mathbbol}
\usepackage{bbm}
\usepackage{stmaryrd}
\usepackage[all]{xy}
\usepackage[usenames,dvipsnames]{xcolor}
\usepackage{tikz}
\usepackage{graphicx}
\usepackage{enumitem}

\definecolor{lightgrey}{rgb}{.804,.804,.756}
\definecolor{myviolet}{rgb}{.45,.05,.545}
\definecolor{myred}{rgb}{.545,0,0}
\definecolor{myblue}{rgb}{.024,.15,.645}
\definecolor{mydarkblue}{rgb}{0,0,.545}
\definecolor{mychoco}{rgb}{.525,.27,.075}
\definecolor{myolive}{rgb}{0,.455,0}

\usepackage[pdfauthor={V. Lebed, F. Wagemann},
pdftitle={Representations of crossed modules and other generalized Yetter-Drinfel'd modules},%
colorlinks=false,linkbordercolor=lightgrey,citebordercolor=lightgrey,urlbordercolor=lightgrey]{hyperref}
\usepackage{caption}
\usepackage[all]{hypcap}

\newcommand*{\C}{\mathcal C}
\newcommand*{\YD}{\mathcal{Y}\mathcal{D}}
\newcommand*{\YDn}{\prescript{\operatorname{n}}{}{\mathcal{Y}\mathcal{D}}}
\newcommand*{\YDnn}{\prescript{\operatorname{nn}}{}{\mathcal{Y}\mathcal{D}}}
\newcommand*{\MMM}{\mathcal{M}}
\newcommand*{\ZZZ}{\mathcal{Z}}
\newcommand*{\g}{\mathfrak{g}}
\renewcommand*{\k}{\mathfrak{k}}
\newcommand*{\gp}{\mathfrak{g}^+}
\newcommand*{\kp}{\mathfrak{k}^+}
\newcommand*{\Ob}{\operatorname{Ob}}
\newcommand*{\Aut}{\operatorname{Aut}}
\newcommand*{\End}{\operatorname{End}}
\renewcommand*{\Ob}{\operatorname{Ob}}
\newcommand*{\Mor}{\operatorname{Mor}}
\newcommand*{\Id}{\operatorname{Id}}
\newcommand*{\II}{\mathbf{I}}
\newcommand*{\ModCat}{\mathbf{Mod}}
\newcommand*{\Modn}{\mathbf{Mod}^{\operatorname{n}}}
\newcommand*{\Modnn}{\mathbf{Mod}^{\operatorname{nn}}}
\newcommand*{\Modasnn}{\mathbf{Mod}^{\operatorname{as,nn}}}
\newcommand*{\CoModnn}{\mathbf{Mod}_{\operatorname{nn}}}
\newcommand*{\CoModn}{\mathbf{Mod}_{\operatorname{n}}}
\newcommand*{\Set}{\mathbf{Set}}
\newcommand*{\kk}{\Bbbk}
\newcommand*{\ZZ}{\mathbb{Z}}
\newcommand*{\Vect}{\mathbf{Vect}_{\kk}}
\newcommand*{\vect}{\mathbf{vect}_{\kk}}
\newcommand*{\Dr}{\mathcal{D}}
\newcommand*{\wop}{\mathrel{\ensuremath{\widetilde{\lhd}}}}
\newcommand*{\op}{\mathrel{\ensuremath{\lhd}}}
\newcommand*{\mop}{\mathrel{\ensuremath{\blacktriangleleft}}}
\newcommand*{\wotimes}{\mathrel{\ensuremath{\widetilde{\otimes}}}}
\newcommand*{\wk}{k'}
\newcommand*{\wm}{m'}
\renewcommand*{\wr}{r'}
\newcommand*{\ws}{s'}
\newcommand*{\wv}{v'}
\newcommand*{\wg}{g'}
\newcommand*{\oV}{\overline{V}}
\newcommand*{\osigma}{\overline{\sigma}}

\newtheorem{proposition}{Proposition}[section]

\newtheorem{theorem}{Theorem}
\newtheorem{lemma}[proposition]{Lemma}
\theoremstyle{definition}
\newtheorem{definition}[proposition]{Definition}
\newtheorem{notation}[proposition]{Notation}
\newtheorem{remark}[proposition]{Remark}
\newtheorem{example}[proposition]{Example}

\begin{document}
\title[Generalized Yetter-Drinfel$'$d modules]{Representations of crossed modules\\ and other generalized Yetter-Drinfel$'$d modules}
\author{Victoria Lebed}
\address{Laboratoire de Math\'ematiques Jean Leray, Universit\'e de Nantes, 2, rue de La Houssini\`ere, 44322 Nantes cedex 3, France}
\email{lebed.victoria@gmail.com}

\author{Friedrich Wagemann}
\address{Laboratoire de Math\'ematiques Jean Leray, Universit\'e de Nantes, 2, rue de La Houssini\`ere, 44322 Nantes cedex 3, France}
\email{wagemann@math.univ-nantes.fr}

\begin{abstract}
The Yang-Baxter equation plays a fundamental role in various areas of mathematics. Its solutions, called braidings, are built, among others, from Yetter-Drinfel$'$d modules over a Hopf algebra, from self-distributive structures, and from crossed modules of groups. In the present paper these three sources of solutions are unified inside the framework of Yetter-Drinfel$'$d modules over a braided system. A systematic construction of braiding structures on such modules is provided. Some general categorical methods of obtaining such generalized Yetter-Drinfel$'$d (=GYD) modules are described. Among the braidings recovered using these constructions are the Woronowicz and the Hennings braidings on a Hopf algebra. We also introduce the notions of crossed modules of shelves / Leibniz algebras, and interpret them as GYD modules. This yields new sources of braidings. We discuss whether these braidings stem from a braided monoidal category, and discover several non-strict pre-tensor categories with interesting associators. 
\end{abstract}

\keywords{Yang-Baxter equation, Yetter-Drinfel$'$d module, Hopf algebra, self-distributivity, crossed module of groups, crossed module of racks, crossed module of Lie algebras, braided system, monoidal category, associator}

\subjclass[2010]{16T25, 
16T05, 
17A32, 
20N02, 
18D10 
} 

\maketitle 


\section{Introduction}

A \textit{Yang-Baxter operator}, or a \textit{braiding}, is a map $\sigma \colon V \otimes V \to V \otimes V$ providing a solution to the \textit{Yang-Baxter equation}
\begin{align}\label{eqn:YBE}\tag{YBE}
(\sigma\otimes V)\circ (V  \otimes\sigma)\circ (\sigma\otimes V) &=
(V\otimes\sigma)\circ (\sigma \otimes V)\circ (V\otimes\sigma);
\end{align}
here and below we use notations of type $V := \Id_V$. This equation makes sense in any strict monoidal category, but in this paper we mainly work in the category $\Vect$ of vector spaces over a field~$\kk$ and in the category $\Set$ of sets (with the symbol~$\otimes$ meaning the tensor product over~$\kk$ and the Cartesian product respectively). The term ``braiding'' comes from the graphical interpretation of \eqref{eqn:YBE}, illustrated in Fig.~\ref{pic:YBE}; here the braiding~$\sigma$ is denoted by \begin{tikzpicture}[scale=0.4]
\draw [rounded corners](0,0)--(0,0.25)--(0.4,0.4);
\draw [rounded corners](0.6,0.6)--(1,0.75)--(1,1);
\draw [rounded corners](1,0)--(1,0.25)--(0,0.75)--(0,1);
\end{tikzpicture}, and all diagrams should be read from bottom to top.
\begin{center}
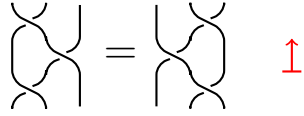

\begin{tikzpicture}[scale=0.45]
\draw [thick,  rounded corners](0,0)--(0,0.25)--(0.4,0.4);
\draw [thick,  rounded corners](0.6,0.6)--(1,0.75)--(1,1.25)--(1.4,1.4);
\draw [thick,  rounded corners](1.6,1.6)--(2,1.75)--(2,3);
\draw [thick,  rounded corners](1,0)--(1,0.25)--(0,0.75)--(0,2.25)--(0.4,2.4);
\draw [thick,  rounded corners](0.6,2.6)--(1,2.75)--(1,3);
\draw [thick,  rounded corners](2,0)--(2,1.25)--(1,1.75)--(1,2.25)--(0,2.75)--(0,3);
\node  at (3.2,1.5){\Large $=$};
\end{tikzpicture}
\begin{tikzpicture}[scale=0.45]
\draw [thick,  rounded corners](1,1)--(1,1.25)--(1.4,1.4);
\draw [thick,  rounded corners](1.6,1.6)--(2,1.75)--(2,3.25)--(1,3.75)--(1,4);
\draw [thick,  rounded corners](0,1)--(0,2.25)--(0.4,2.4);
\draw [thick,  rounded corners](0.6,2.6)--(1,2.75)--(1,3.25)--(1.4,3.4);
\draw [thick,  rounded corners](1.6,3.6)--(2,3.75)--(2,4);
\draw [thick,  rounded corners](2,1)--(2,1.25)--(1,1.75)--(1,2.25)--(0,2.75)--(0,4);
\draw [|->, red, thick]  (4,2) -- (4,3);
\end{tikzpicture}
\captionof{figure}{The YBE as the third Reidemeister move}\label{pic:YBE}
\end{center}

The YBE plays a fundamental role in such apparently distant fields as statistical mechanics, particle physics, quantum field theory, quantum group theory, and low-dimensional topology; see for instance \cite{YBE} for a brief introduction. The study of its solutions has been a vivid research area for the last half of a century. Two sources of braidings proved to be of particular importance:

\begin{description}[leftmargin=7mm]
\item[Source 1]\label{ex:YD} A \textit{(right-right)\footnote{In this paper all the (co)actions are on the right, so the term \emph{right} is systematically omitted in what follows.} Yetter-Drinfel$'$d module} over a Hopf algebra~$H$ is a vector space~$M$ endowed with a right  $H$-action~$\rho$ and a right $H$-coaction~$\delta$, compatible in the following sense:
\begin{align}\label{eqn:YD}
(m \ast h)_{(0)} \otimes (m \ast h)_{(1)} &= m_{(0)} \ast h_{(2)} \otimes s(h_{(1)})m_{(1)} h_{(3)}
\end{align}
(we use the symbol~$\ast$ for the action~$\rho$, and M.E.~Sweedler's formal notations $\delta(m)= m_{(0)} \otimes m_{(1)}$, $\Delta(h)= h_{(1)} \otimes h_{(2)}$, with the summation sign omitted). These structures were introduced by D.~Yetter~\cite{Yetter} under the name ``crossed bimodules'', and repeatedly rediscovered under different names. They are known to be at the origin of a very vast family of invertible braidings, which is complete in the category $\vect$ of finite-dimensional vector spaces \cite{FRT2,FRT1,Rad2}. Concretely, the map 
\begin{align}\label{eqn:br_YD}
\sigma_{YD}(m \otimes n)& = n_{(0)} \otimes m \ast n_{(1)}
\end{align}
endows a YD module~$M$ with a braiding. One can say more: the category~$\YD_H^H$ of YD modules over~$H$ is braided monoidal, and even modular when~$H$ is a group algebra of a finite group. This rich categorical structure is at the heart of powerful invariants of links and $3$-dimensional manifolds.

\item[Source 2]\label{ex:SD} A \textit{self-distributive set}, or briefly \textit{shelf}, is a set~$S$ endowed with a binary operation~$\op$ which is self-distributive, in the sense of
\begin{align}\label{eqn:SD}
(a \op b) \op c = (a \op c) \op (b \op c).
\end{align}
Major examples are
\begin{itemize}
\item groups with the conjugation operation $g \op g' = (g')^{-1} gg'$;
\item sets~$S$ with a preferred map $f \colon S \to S$, their shelf operation defined by $s \op s' = f(s)$.
\end{itemize}
A shelf carries the following braiding: 
\begin{align}\label{eqn:br_SD}
\sigma_{SD}(a,b) = (b, a \op b),
\end{align}
which is the key ingredient of an extremely strong and efficiently computable class of invariants of links, knotted surfaces, knotted graphs, and other topological objects. The self-distributive approach to knot theory originated from the work of D.~Joyce and S.V.~Matveev \cite{Joyce,Matveev}; see also~\cite{Crans} for a formulation in terms of braidings.
\end{description}

A new source of braidings was recently found by P.~Bantay~\cite{Bantay}:
\begin{description}[leftmargin=7mm]
\item[Source 3]\label{ex:CrMod} A \textit{crossed module of groups} is the data of a group morphism $\pi \colon K \to G$ and a (right) $G$-action~$\cdot$ on~$K$ by group automorphisms, compatible in the sense of
\begin{align}
k \cdot \pi(\wk) &= (\wk)^{-1}k\wk, & &k,\wk \in K, \label{eqn:CrModGr}\\
\pi (k \cdot g) &= g^{-1} \pi(k) g, & &k \in K,\, g \in G.\label{eqn:CrModGr'}
\end{align}
A $K$-graded $G$-module $(M=\oplus_{k \in K} M_k,\ast)$ with the action of any $g \in G$ sending $M_k$ onto $M_{k \cdot g}$ is called a \textit{representation} of $(K,G,\pi,\cdot)$. The map
\begin{align}\label{eqn:br_CrMod}
\sigma_{CrMod}(m \otimes n) = \sum\nolimits_{k \in K} n_k \otimes m \ast \pi(k)
\end{align}
defines a braiding on such an~$M$; here $n_k$ is the component of~$n$ living in the grading~$k$. Again, there is much more structure in the story: the representations of a crossed module form a braided monoidal category $\MMM(K,G,\pi,\cdot)$ (often abusively denoted by $\MMM(K,G)$), which is pre-modular if~$G$ and~$K$ are finite, and modular if moreover $\pi$ is a group isomorphism (in which case one recovers the category $\YD_{\kk G}^{\kk G}$). See \cite{Bantay,MaiSch} for more details.
\end{description}

In these three cases, the braidings share the same form:
\begin{align}\label{eqn:br_genYD_intro}
\sigma& = (M \otimes \rho) \circ (\tau \otimes \pi) \circ  (M \otimes \delta)
\end{align}
 (see Fig.~\ref{eqn:br_YD} for a graphical version). Here~$\tau$ is the \textit{flip} 
 $$\tau(a \otimes b) = b \otimes a;$$ 
 $\pi$ is the identity in the first two examples; for a shelf we put $\delta(a)=a \otimes a$ and $\rho(a \otimes b) = a \lhd b$; and in the last example, $\delta(m)=\sum_{k \in K} m_k \otimes k$.

\begin{center}
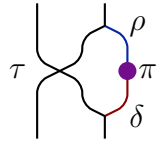

\begin{tikzpicture}[scale=0.6]
\draw [myred, thick,  rounded corners] (1.5,1.5)--(2,1.75)--(2,2.5);
\draw [myblue, thick,  rounded corners] (2,2.5)--(2,3.25)--(1.5,3.5);
\draw [thick,  rounded corners] (0,1)--(0,2.25)--(1,2.75)--(1,3.25)--(1.5,3.5);
\draw [thick]  (1.5,3.5)--(1.5,4);
\draw [thick,  rounded corners] (1.5,1.5)--(1,1.75)--(1,2.25)--(0,2.75)--(0,4);
\draw [thick] (1.5,1.5)--(1.5,1);
\fill [myviolet] (2,2.5) circle (0.2);
\node [right]  at (2,2.5){$\pi$};
\node [right]  at (1.8,1.5){$\delta$};
\node [right]  at (1.8,3.5){$\rho$};
\node [left]  at (0,2.5){$\tau$};
\end{tikzpicture}
\captionof{figure}{The general form of braidings}\label{pic:br_genYD}
\end{center}

In this paper we introduce the category~$\YD_A^C$ of \textit{generalized Yetter-Drin\-fel$'$d modules}, where~$A$ and~$C$ are \textit{braided objects} (i.e., objects endowed with braidings) in a symmetric monoidal category~$\C$, related via an \textit{entwining map} $C \otimes A \to A \otimes C$ (Section~\ref{sec:YDMod}). 
Under certain conditions on the map $\pi \colon C \to A$, formula~\eqref{eqn:br_genYD_intro} (with the map~$\tau$ replaced with the underlying categorical braiding of~$\C$) yields a braiding\footnote{The Reader should be careful with the various braidings entering our construction at different levels: the global symmetric braiding defined on the whole category~$\C$, the local braidings the objects~$A$ and~$C$ come with, and the braiding~$\sigma$ we aim to constructs for our generalized YD module.} on any generalized YD module~$M$ (Theorem~\ref{thm:genYD_br}). 
This abstract setting unifies the three braiding constructions above.  
Sections \ref{sec:RepsCrModSh}-\ref{sec:RepsCrModLA} treat some original ones, based on
\begin{itemize}
\item \textit{twisted crossed modules of shelves}, which generalize \textit{crossed modules of racks}, defined by A.~Crans and the second author~\cite{CransWag}; 
\item \textit{non-normalized crossed modules of Leibniz algebras}, which generalize classical \textit{crossed modules of Lie algebras}. 
\end{itemize}
In particular, we introduce the notion of \textit{representations of a crossed module of shelves / Leibniz algebras}, and endow them with braidings (Theorems~\ref{thm:SrModShBr}-\ref{thm:SrModLABr}). 
Section~\ref{sec:AdjMod} describes a vast source of generalized YD modules. At its heart is a categorical-center-like construction, close in spirit to the \textit{factorisations of a distributive law} of U.~Kr\"{a}hmer and P.~Slevin~\cite{KraSle}. Various generalized YD module structures on a Hopf algebra are presented as an illustration, the associated braidings recovering those of S.L.~Woronowicz and M.A.~Hennings \cite{Wor,Hennings}. Possible (pre-)tensor structures on~$\YD_A^C$ are discussed for our major examples. In the case of crossed modules of racks / Leibniz algebras, we discover \textit{pre-tensor categories} (i.e., categories with a tensor product but without a unit object) with interesting non-trivial associativity morphisms (Theorems~\ref{thm:TensorRepCrModSh} and~\ref{thm:TensorRepCrModLA}).

The idea of mixing compatible acting and coacting structures can be found in the literature under various guises: J.~Beck's mixed distributive laws~\cite{Beck}, the $AC$-bialgebras of T.F.~Fox and M.~Markl~\cite{FoxMarkl}, the algebra-coalgebra entwining structures of T.~Brzezi{\'n}ski and S.~Majid~\cite{EntwAlgCoalg}, J.-L.~Loday's generalized bialgebras~\cite{GenBialg} interpreted in terms of bimodules over a bimonad by M.~Livernet, B.~Mesablishvili, and R.~Wisbauer \cite{Bimonads,BialgBimonads}, to cite just a few. 
The framework chosen in each case depends on the classical constructions and results one wants to extend to a generalized setting: the triple-cotriple philosophy of~\cite{FoxMarkl} is well adapted for (co)homology constructions, the category of vector spaces is sufficient for developing quantum principal bundle theory and generalized gauge theory in~\cite{EntwAlgCoalg}, while operads provide a convenient setting for generalizing Poincar\'{e}-Birkhoff-Witt, Cartier-Milnor-Moore, and the Rigidity theorems in~\cite{GenBialg}. The \textit{braided framework} (with a ``braided-distributive'' law relating the action and coaction by braided objects) is adopted in the present paper for the following reasons:
\begin{itemize}
\item as shown in~\cite{Lebed1,Lebed2}, it includes all the basic structures we are interested in: (co)associative (co)algebras, bialgebras, Leibniz algebras, shelves, etc.;
\item it allows one to treat both structural and entwining maps for acting and coacting structures in a uniform way;
\item technical verifications can often be substituted with the more user-friendly and transparent diagrammatic calculus; 
\item the map~\eqref{eqn:br_genYD_intro} is well defined and remains a reasonable candidate for being a braiding on our modules.
\end{itemize}
The connections with the framework of entwining structures are discussed in Remark~\ref{rmk:Entwining}.  
   
\textbf{Acknowledgements.} This work was supported by Henri Lebesgue Centre (University of Nantes), and by the program ANR-11-LABX-0020-01. We thank Peter Schauenburg, Ulrich Kr\"{a}hmer, and Ya\"{e}l Fr\'{e}gier for fruitful discussions. 

\section{Generalized Yetter-Drinfel$'$d modules}\label{sec:YDMod}

Fix a strict monoidal category $(\C,\otimes,\II)$. In order to introduce the notion of generalized Yetter-Drinfel$'$d modules, we first recall some definitions from~\cite{Lebed2}:

\begin{definition}[Braided vocabulary]\label{def:Braided}
$ $

\begin{itemize}
\item  A \emph{rank~$r$ braided system} in~$\C$ is a family $V_1,V_2,\ldots, V_r$ of objects of~$\C$ endowed with a \emph{(multi-)braiding}, i.e., morphisms
 $$\sigma_{i,j} = \sigma_{V_i,V_j} \colon V_i \otimes V_j \to V_j \otimes V_i, \quad 1\le {i  \le j} \le r,$$  
 satisfying the \emph{colored Yang-Baxter equation}
 \begin{align}\label{eqn:cYBE}\tag{cYBE}
 (\sigma_{j,k}\otimes V_i)\circ(V_j \otimes \sigma_{i,k})\circ(\sigma_{i,j}\otimes V_k) &=\\
 (V_k \otimes \sigma_{i,j})\circ(\sigma_{i,k}\otimes V_j)\circ(V_i \otimes \sigma_{j,k})& \notag
\end{align}
on all the tensor products $V_i \otimes V_j \otimes V_k$ with $1\le {i \le j \le k} \le r$. Such a system is denoted by $((V_i)_{1\le i\le r};(\sigma_{i,j})_{1\le i\le j\le r})$ or briefly $(\oV,\osigma)$.
\item Rank~$1$ braided systems are called \emph{braided objects} in~$\C$.
\item A \emph{(right) braided module} over a braided system $(\oV,\osigma)$  is an object~$M$ equipped with morphisms $\overline{\rho}:=(\rho_i:M\otimes V_i \rightarrow M)_{1 \le i \le r}$ satisfying, for all $1 \le i \le j \le r$,
\begin{align}\label{eqn:BrMod}
\rho_j \circ (\rho_i \otimes V_{j}) &=\rho_i \circ (\rho_j \otimes V_{i})\circ (M\otimes \sigma_{i,j}). 
\end{align} 
Here both morphisms go from $M\otimes V_i \otimes V_j$ to~$M$. 
\item A \emph{morphism between braided modules} $(M,\overline{\rho})$ and $(M',\overline{\rho'})$ over $(\oV,\osigma)$ is a morphism $\varphi \in \Mor_\C(M,M')$ respecting the module structures, in the sense of
\begin{align}\label{eqn:BrModMor}
\varphi \circ \rho_i &=  \rho'_i \circ (\varphi \otimes V_{i}). 
\end{align} 
\item The category of braided modules and their morphisms is denoted by $\ModCat_{(\oV,\osigma)}$.
 \emph{(Right) braided comodules}, their morphisms, and the category $\ModCat^{(\oV,\osigma)}$ are defined in an analogous way. 
\item A braided (co)module structure on the unit object~$\II$ is referred to as a \emph{braided (co)character}. 
\end{itemize}
\end{definition}

Note that our braidings are not necessarily invertible.

The defining relations~\eqref{eqn:cYBE} and~\eqref{eqn:BrMod} can be expressed in the language of colored knotted graphs, as shown in Fig.~\ref{pic:cYBE}.
\begin{center}
\begin{tikzpicture}[scale=0.5]
\draw [thick, myred, rounded corners](0,0)--(0,0.25)--(0.4,0.4);
\draw [thick, myred, rounded corners](0.6,0.6)--(1,0.75)--(1,1.25)--(1.4,1.4);
\draw [thick, myred, rounded corners](1.6,1.6)--(2,1.75)--(2,3);
\draw [thick, myblue, rounded corners](1,0)--(1,0.25)--(0,0.75)--(0,2.25)--(0.4,2.4);
\draw [thick, myblue, rounded corners](0.6,2.6)--(1,2.75)--(1,3);
\draw [thick, myviolet, rounded corners](2,0)--(2,1.25)--(1,1.75)--(1,2.25)--(0,2.75)--(0,3);
\node  at (0,0) [myred, below] {$V_i$};
\node  at (1,0) [myblue, below] {$V_j$};
\node  at (2,0) [myviolet, below] {$V_k$};
\node  at (3.5,1.5){$=$};
\end{tikzpicture}
\begin{tikzpicture}[scale=0.5]
\node  at (-0.5,1.5){};
\draw [thick, myblue, rounded corners](1,1)--(1,1.25)--(1.4,1.4);
\draw [thick, myblue, rounded corners](1.6,1.6)--(2,1.75)--(2,3.25)--(1,3.75)--(1,4);
\draw [thick, myred, rounded corners](0,1)--(0,2.25)--(0.4,2.4);
\draw [thick, myred, rounded corners](0.6,2.6)--(1,2.75)--(1,3.25)--(1.4,3.4);
\draw [thick, myred, rounded corners](1.6,3.6)--(2,3.75)--(2,4);
\draw [thick, myviolet, rounded corners](2,1)--(2,1.25)--(1,1.75)--(1,2.25)--(0,2.75)--(0,4);
\node  at (0,1) [myred, below] {$V_i$};
\node  at (1,1) [myblue, below] {$V_j$};
\node  at (2,1) [myviolet, below] {$V_k$};
\end{tikzpicture}
\begin{tikzpicture}[scale=0.5]
 \node at (-5,1) { };
 \draw [ultra thick, myolive] (0,0) -- (0,3);
 \draw [thick, myred] (1,0) -- (0,1);
 \draw [thick, myblue] (2,0) -- (0,2);
 \node at (0,2) [left]{$\rho_{\color{myblue} j}$};
 \node at (0,1) [left]{$\rho_{\color{myred} i}$};
 \fill[myblue] (0,2) circle (0.2);
 \fill[myred] (0,1) circle (0.2); 
 \node at (2,0) [myblue,below] {$V_j$};
 \node at (1,0) [myred,below] {$V_i$};
 \node at (0,0) [myolive,below] {$M$};
 \node  at (3.5,1.5){$=$};
\end{tikzpicture}
\begin{tikzpicture}[scale=0.5]
\node  at (-1,1.5){};
 \draw [myolive, ultra thick] (0,0) -- (0,3);
 \draw [thick, myred] (1,0) -- (0.87,0.35);
 \draw [thick, myred] (0.55,0.9) -- (0,2);
 \draw [thick, myblue] (2,0) -- (0,1);
 \node at (0,1) [left]{$\rho_{\color{myblue} j}$};
 \node at (0,2) [left]{$\rho_{\color{myred} i}$};
 \fill[myred] (0,2) circle (0.2);
 \fill[myblue] (0,1) circle (0.2); 
 \node at (2,0) [myblue,below] {$V_j$};
 \node at (1,0) [myred,below] {$V_i$};
 \node at (0,0) [myolive,below] {$M$};
 \node at (1.5,1.2) {$\sigma_{{\color{myred} i},{\color{myblue} j}}$}; 
\end{tikzpicture}

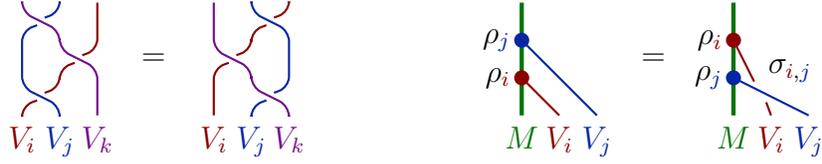
\captionof{figure}{Braided systems and braided modules}\label{pic:cYBE}
\end{center}

The following basic examples from~\cite{Lebed1} will be used in what follows:

\begin{example}[Unital associative algebras]\label{ex:Ass}
$ $

A unital associative algebra $(A,\mu,\nu)$ in~$\C$ carries the braiding
\begin{align}\label{eqn:br_Ass}
\sigma_{Ass} = \nu \otimes \mu,
\end{align}
which in the category $\Vect$ becomes $\sigma_{Ass} (v \otimes \wv) = 1 \otimes v \wv$. The YBE for $\sigma_{Ass}$ is equivalent to the associativity of~$\mu$. The notion of braided module over $(A;\sigma_{Ass})$ is slightly broader that the usual notion of module over the algebra~$A$: it involves the mixed relation 
\begin{align*}
\rho \circ (\rho \otimes A) \circ (M \otimes \nu \otimes \mu) &= \rho \circ (\rho \otimes A)
\end{align*}
instead of the usual separate relations 
\begin{align*}
\rho \circ (M \otimes \mu) &= \rho \circ (\rho \otimes A), & \rho \circ (M \otimes \nu) &= M.
\end{align*}
The category of such \emph{non-normalized algebra modules}, with the usual notion of morphisms, is denoted by $\Modnn_A$. The notation $\Modn_A$ is reserved for usual, or normalized, algebra modules.
\end{example}

\begin{example}[Counital coassociative coalgebras]\label{ex:coAss}
$ $

Dually, a counital coassociative algebra $(C,\Delta,\varepsilon)$ in~$\C$ is a braided object, with the braiding
\begin{align}\label{eqn:br_coAss}
\sigma_{coAss} = \varepsilon \otimes \Delta,
\end{align}
or, in $\Vect$, $\sigma_{coAss} (v \otimes \wv) = \varepsilon(v) \Delta(\wv)$. The YBE for $\sigma_{coAss}$ is equivalent to the coassociativity of~$\Delta$. Similarly to the algebra case, braided comodules over $(C;\sigma_{coAss})$ form a category $\CoModnn^C$ which extends the category $\CoModn^C$ of usual $C$-comodules.
\end{example}

\begin{example}[Shelves]\label{ex:SDbr}
$ $

A shelf $(S,\op)$ is a braided object in~$\Set$, with the braiding~$\sigma_{SD}$ from~\eqref{eqn:br_SD}. The YBE is equivalent to the self-distributivity~\eqref{eqn:SD} here, and braided modules over $(S;\sigma_{SD})$ are the usual modules over~$S$ (also called \textit{$S$-sets}, or \textit{$S$-shadows}; see Definition~\ref{def:CrModSh} for details). 
\end{example}

\begin{example}[Leibniz algebras]\label{ex:Lei}
$ $

Recall that a \textit{right (unital) Leibniz algebra} in a symmetric preadditive monoidal category~$(\C,\otimes,\II,c)$ is an object~$\g$ with morphisms $[,] \colon \g\otimes \g\to \g$ (and $\nu \colon \II\to \g$) satisfying the \textit{Leibniz} (and the \textit{Lie unit}) \textit{conditions}
\begin{align*}
[,] \circ(\g \otimes [,]) &= [,]\circ([,] \otimes \g)-[,]\circ([,] \otimes \g)\circ(\g \otimes c_{\g,\g}),\\ 
[,] \circ(\g\otimes\nu) &= [,] \circ(\nu \otimes \g) =0,
\end{align*}
which in~$\Vect$ become $[v,[w,u]]=[[v,w],u]-[[v,u],w]$ and $[v,1] = [1,v] = 0$. An example of a unital Leibniz algebra in~$\Vect$ is given by the endomorphism algebra $\End_{\kk}(M)$ of a vector space~$M$, with $[f,g] = fg-gf$ and $1 = \Id_M$. This generalization of Lie algebras appeared, in its non-unital version, in the work of C.~Cuvier and J.-L.~Loday \cite{Cuvier,Cyclic,LoLei}. 
To such data one can associate the braiding
\begin{align}\label{eqn:Lei}
\sigma_{Lei} = c_{\g,\g}+ \nu \otimes [,],
\end{align}
which in~$\Vect$ reads $\sigma_{Lei} (v \otimes \wv) = \wv \otimes v + 1 \otimes [v, \wv]$. The YBE for $\sigma_{Lei}$ is equivalent to the Leibniz condition for~$[,]$. Braided modules over $(\g;\sigma_{Lei})$ are identified with non-normalized anti-symmetric modules over our Leibniz algebra; the corresponding module category is denoted by $\Modasnn_{\g}$. See~\cite{LoPi} for more details on the representation theory of Leibniz algebras.
\end{example}

\begin{example}[Bialgebras]\label{ex:HopfYD}
$ $

Take a finite-dimensional Hopf algebra~$H$ over~$\kk$. Two braided system structures on $(H,H^*)$ were described in \cite{Lebed,Lebed2}. Braided modules over these systems include, respectively, Hopf modules and YD modules over~$H$. A rank~$4$ braided system from~\cite{Lebed2} allows one to recover Hopf bimodules. 
\end{example}

\begin{remark}
A rank~$2$ braided system $(C,A;\sigma_{C,C},\sigma_{A,A},\sigma_{C,A})$ decomposes as two braided objects $(C;\sigma_{C,C})$ and $(A;\sigma_{A,A})$, connected by an \textit{entwining map} $C \otimes A \to A \otimes C$  satisfying two compatibility conditions, namely, \eqref{eqn:cYBE} on $C\otimes C \otimes A$ and on $C\otimes A \otimes A$.
\end{remark}

\begin{definition}[Generalized YD modules]\label{def:gen_YD}
$ $

\begin{itemize}
\item Let $(C,A;\sigma_{C,C},\sigma_{A,A},\sigma_{C,A})$ be a braided system in~$\C$. A \emph{(right-right) Yetter-Drinfel$'$d module} over this system is an object~$M$ of~$\C$ with a right $(A;\sigma_{A,A})$-module structure~$\rho$ and a right $(C;\sigma_{C,C})$-comodule structure~$\delta$, compatible in the sense of
\begin{align}\label{eqn:gen_YD}
\delta \circ \rho &= (\rho \otimes C) \circ (M \otimes \sigma_{C,A}) \circ (\delta \otimes A).
\end{align}
\item A morphism between two YD modules over the same braided system is a morphism in~$\C$ preserving the module and the comodule structures (cf.~\eqref{eqn:BrModMor}).
\item The category of YD modules over $(C,A;\osigma)$ and their morphisms is denoted by~$\YD_A^C$.
\end{itemize}
\end{definition}

Condition~\eqref{eqn:gen_YD} is graphically represented in Fig.~\ref{pic:gen_YD}. It can be interpreted as the requirement for~$\delta$ to be a morphism in $\ModCat_{(A;\sigma_{A,A})}$, or equivalently the requirement for~$\rho$ to be a morphism in $\ModCat^{(C;\sigma_{C,C})}$ (Remark~\ref{rmk:ActCoactCompat}). It can also be regarded as a \textit{braided-distributive law}, which allows the action and the coaction to switch places in a composition with the help of the entwining braiding component.  
In~$\Vect$, \eqref{eqn:gen_YD} becomes
\begin{align*}
(m \ast a)_{(0)} \otimes (m \ast a)_{(1)} &= m_{(0)} \ast \widetilde{a} \otimes \widetilde{m_{(1)}},
\end{align*}
using Sweedler's notations and another formal notation $\sigma_{C,A}(c \otimes a)= \widetilde{a} \otimes \widetilde{c}$.

\begin{center}
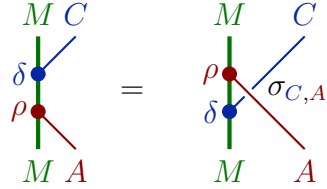

\begin{tikzpicture}[scale=0.5]
 \draw [ultra thick, myolive] (0,0) -- (0,3);
 \draw [thick, myred] (1,0) -- (0,1);
 \draw [thick, myblue] (1,3) -- (0,2);
 \node at (0,2) [left]{$\color{myblue} \delta$};
 \node at (0,1) [left]{$\color{myred} \rho$};
 \fill[myblue] (0,2) circle (0.2);
 \fill[myred] (0,1) circle (0.2); 
 \node at (1,3) [myblue,above] {$C$};
 \node at (1,0) [myred,below] {$A$};
 \node at (0,0) [myolive,below] {$M$};
 \node at (0,3) [myolive,above] {$M$}; 
 \node  at (2.5,1.5){$=$};
\end{tikzpicture}
\begin{tikzpicture}[scale=0.5]
\node  at (-1.5,1.5){};
 \draw [myolive, ultra thick] (0,0) -- (0,3);
 \draw [thick, myred] (2,0) -- (0,2);
 \draw [thick, myblue] (2,3) -- (0.6,1.6);
 \draw [thick, myblue]  (0.4,1.4) -- (0,1);
 \node at (0,1) [left]{$\color{myblue} \delta$};
 \node at (0,2) [left]{$\color{myred} \rho$};
 \fill[myred] (0,2) circle (0.2);
 \fill[myblue] (0,1) circle (0.2); 
 \node at (2,3) [myblue,above] {$C$};
 \node at (2,0) [myred,below] {$A$};
 \node at (0,0) [myolive,below] {$M$};
 \node at (0,3) [myolive,above] {$M$}; 
 \node at (1.8,1.5) {$\sigma_{{\color{myblue} C},{\color{myred} A}}$}; 
\end{tikzpicture}
\captionof{figure}{Generalized YD modules}\label{pic:gen_YD}
\end{center}

\begin{remark}[An alternative viewpoint: entwining structures]\label{rmk:Entwining}
$ $

In a sufficiently nice category (for instance in~$\Vect$ or~$\Set$), an $(A;\sigma_{A,A})$-module structure is the same thing as a module structure over the unital associative algebra
$$T_{\sigma}(A) =  \raisebox{1mm}{$T(A)$} \big/ \raisebox{-1mm}{$\langle \sigma_{A,A} - \Id_{A \otimes A}\rangle$},$$
where $T(A)$ is the tensor algebra of~$A$, $\langle \sigma_{A,A} - \Id_{A \otimes A}\rangle$ is its two-sided ideal generated by the image of $\sigma_{A,A} - \Id_{A \otimes A}$, and the product on~$T_{\sigma}(A)$ is induced by the concatenation. Dually, a $(C;\sigma_{C,C})$-comodule can be regarded as a comodule over the counital coassociative coalgebra $T_{\sigma}(C)$, with the coproduct induced by the deconcatenation. Further, the entwining map $\sigma_{C,A}$ extends to a map $\sigma_{T(C),T(A)}$ in the standard way, which then descends to a map $\sigma_{T_{\sigma}(C),T_{\sigma}(A)}$ since $\sigma_{C,A}$ respects $\sigma_{A,A}$ and $\sigma_{C,C}$ (in the sense of~\eqref{eqn:cYBE}). One obtains an entwining map $\sigma_{T_{\sigma}(C),T_{\sigma}(A)}$ between the algebra~$T_{\sigma}(A)$ and the coalgebra~$T_{\sigma}(C)$, in the sense of~\cite{EntwAlgCoalg}. Moreover, a comparison of the respective action-coaction compatibility conditions yields a category isomorphism
$$\YD_A^C \simeq \ModCat_{T_{\sigma}(A)}^{T_{\sigma}(C)},$$
with the category of $T_{\sigma}(A)$-$T_{\sigma}(C)$-bimodules in the sense of~\cite{FoxMarkl} on the right. On the other hand, such algebra-coalgebra-bimodules form a particular case of our generalized YD modules, since (co)associative structures can be regarded as braided ones, as described in Examples \ref{ex:Ass}-\ref{ex:coAss}. In what follows we stick to the braided approach, more efficient and convenient for our goals. The algebra-coalgebra viewpoint will only re-emerge in Examples~\ref{ex:AssGroup} and~\ref{ex:CrModRAsGr}.
\end{remark}

We now show how to endow YD modules over $(C,A;\osigma)$ with a braiding, provided that our category~$\C$ is symmetric, and our braided system comes with a ``nice'' \textit{connecting morphism} $\pi \colon C \to A$.

\begin{theorem}[Braiding for generalized YD modules]\label{thm:genYD_br}
$ $

Take a braided system $(C,A;\osigma)$ and a morphism $\pi \colon C \to A$ in a symmetric strict monoidal category $(\C,\otimes,\II,c)$. Suppose that for some non-negative integers $\alpha_1, \alpha_2, \gamma_1, \gamma_2$ the following technical condition is satisfied:
\begin{align}
&(A \otimes \sigma_{A,A}^{\alpha_1}) \circ (c_{A,A} \otimes \pi) \circ (\pi \otimes \sigma_{C,A}) \circ (c_{C,C} \otimes \pi) \circ (C \otimes \sigma_{C,C}^{\gamma_1}) \label{eqn:cond_pi}  \\
 = (A &\otimes \sigma_{A,A}^{\alpha_2}) \circ (A \otimes \pi \otimes A) \circ (c_{C,A} \otimes \pi) \circ (C \otimes \pi \otimes C) \circ (C \otimes \sigma_{C,C}^{\gamma_2}) \notag
\end{align}
(Fig.~\ref{pic:cond_pi}). Then any Yetter-Drinfel$'$d modules $(M_i,\rho_i,\delta_i)$ over $(C,A;\osigma)$ 
 form a braided system in~$\C$, with the braiding on $M_i \otimes M_j$ defined by
\begin{align}\label{eqn:br_genYD}
\sigma_{gYD} = ({M_j} \otimes \rho_i) \circ (c_{M_i, M_j} \otimes \pi) &\circ ({M_i} \otimes \delta_j).
\end{align}
\end{theorem}

\begin{center}
\begin{tikzpicture}[scale=1]
\end{tikzpicture}
\begin{tikzpicture}[scale=0.5]
\draw [thick, myblue, rounded corners](0,-1)--(0,0.25)--(1,0.75)--(1,1.25)--(1.4,1.4);
\draw [thick, myblue, rounded corners](1.6,1.6)--(2,1.75)--(2,2.5);
\draw [thick, myred, rounded corners](2,2.5)--(2,4);
\draw [thick, myblue, rounded corners](1,-1)--(1,0.25)--(0,0.75)--(0,1.5);
\draw [thick, myred, rounded corners](0,1.5)--(0,2.25)--(1,2.75)--(1,4);
\draw [thick, myblue, rounded corners](2,-1)--(2,0.5);
\draw [thick, myred, rounded corners](2,0.5)--(2,1.25)--(1,1.75)--(1,2.25)--(0,2.75)--(0,4);
\node  at (0,-1) [myblue, below] {$C$};
\node  at (1,-1) [myblue, below] {$C$};
\node  at (2,-1) [myblue, below] {$C$};
\node  at (0,4) [myred, above] {$A$};
\node  at (1,4) [myred, above] {$A$};
\node  at (2,4) [myred, above] {$A$};
\node at (2,1.5) [right] {$\sigma_{{\color{myblue} C},{\color{myred} A}}$}; 
\node at (0,0.5) [left] {$c_{{\color{myblue} C},{\color{myblue} C}}$}; 
\node at (0,2.5) [left] {$c_{{\color{myred} A},{\color{myred} A}}$}; 
\node at (2,-0.5) [right] {$\sigma_{{\color{myblue} C},{\color{myblue} C}}^{\gamma_1}$}; 
\node at (2,3.5) [right] {$\sigma_{{\color{myred} A},{\color{myred} A}}^{\alpha_1}$}; 
\node at (2,0.5) [right] {$\pi$};
\node at (2,2.5) [right] {$\pi$};
\node at (0,1.5) [left] {$\pi$};
\fill [myviolet] (2,0.5) circle (0.1);
\fill [myviolet] (0,1.5) circle (0.1);
\fill [myviolet] (2,2.5) circle (0.1);
\fill [myblue] (1,-0.75) rectangle (2,-0.25);
\fill [myred] (1,3.25) rectangle (2,3.75);
\node  at (6.5,1.5){$=$};
\end{tikzpicture}
\begin{tikzpicture}[scale=0.5]
\node  at (-2.5,1.5){};
\draw [thick, myblue, rounded corners](2,0)--(2,2.5);
\draw [thick, myred, rounded corners](2,2.5)--(2,5);
\draw [thick, myblue, rounded corners](0,0)--(0,2.25)--(1,2.75)--(1,3.5);
\draw [thick, myred, rounded corners](1,3.5)--(1,5);
\draw [thick, myblue, rounded corners](1,0)--(1,1.5);
\draw [thick, myred, rounded corners](1,1.5)--(1,2.25)--(0,2.75)--(0,5);
\node  at (0,0) [myblue, below] {$C$};
\node  at (1,0) [myblue, below] {$C$};
\node  at (2,0) [myblue, below] {$C$};
\node  at (0,5) [myred, above] {$A$};
\node  at (1,5) [myred, above] {$A$};
\node  at (2,5) [myred, above] {$A$};
\node at (0,2.5) [left] {$c_{{\color{myblue} C},{\color{myred} A}}$}; 
\node at (2,0.5) [right] {$\sigma_{{\color{myblue} C},{\color{myblue} C}}^{\gamma_2}$}; 
\node at (2,4.5) [right] {$\sigma_{{\color{myred} A},{\color{myred} A}}^{\alpha_2}$}; 
\node at (0.8,3.5) [right] {$\pi$};
\node at (2,2.5) [right] {$\pi$};
\node at (0.8,1.5) [right] {$\pi$};
\fill [myviolet] (1,3.5) circle (0.1);
\fill [myviolet] (1,1.5) circle (0.1);
\fill [myviolet] (2,2.5) circle (0.1);
\fill [myblue] (1,0.25) rectangle (2,0.75);
\fill [myred] (1,4.25) rectangle (2,4.75);
\end{tikzpicture}

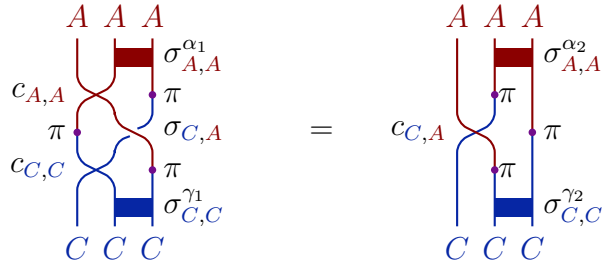
\captionof{figure}{Technical condition on the connecting morphism}\label{pic:cond_pi}
\end{center}

Fig.~\ref{pic:br_genYD} contains a diagrammatic version of this braiding. 
In~$\Vect$, it can be written using Sweedler's notations as
\begin{align*}
\sigma_{gYD}(m \otimes n)& = n_{(0)} \otimes m \ast \pi(n_{(1)}).
\end{align*}

\begin{remark}
Note that in concrete examples, $\pi$ is rarely a morphism of braided objects. 
\end{remark}

\begin{remark}\label{rmk:idempot}
In the examples we are interested in, the morphisms $\sigma_{C,C}$ and/or $\sigma_{A,A}$ are often idempotent. In this case, condition~\eqref{eqn:cond_pi} holds for all sufficiently large values of the $\gamma$'s and/or the $\alpha$'s.
\end{remark}

\begin{proof}[Proof of Theorem~\ref{thm:genYD_br}]
We verify Equation~\eqref{eqn:cYBE} for~$\sigma_{gYD}$ graphically. The naturality and the symmetry of the categorical braiding~$c$ (depicted by a solid crossing 
\,\begin{tikzpicture}[xscale=0.4,yscale=0.3]
\draw [rounded corners](0,0)--(0,0.25)--(1,0.75)--(1,1);
\draw [rounded corners](1,0)--(1,0.25)--(0,0.75)--(0,1);
\end{tikzpicture}\,) is repeatedly used: it allows one to move any strand across any part of a diagram.
The desired equation is depicted on Fig.~\ref{pic:proof}.
\begin{center}
\begin{tikzpicture}[scale=1]
\draw [ultra thick,myolive,rounded corners](0,0)--(0,0.25)--(1,0.75)--(1,1.25)--(2,1.75)--(2,3) ;
\draw (0,0) node[below]{$M_i$}; \draw (2,0) node[below]{$M_k$};
\draw (1,0) node[below]{$M_j$}; 
\draw (1.3,0.5) node[right]{{\tiny $\pi$}}; \draw (1,0.2) node[right]{{\tiny $C$}};
\draw (1,0.8) node[right]{{\tiny $A$}};
\draw [ultra thick,myolive,rounded corners](1,0)--(1,0.25)--(0,0.75)--(0,2.25)--(1,2.75)--(1,3);
\draw [ultra thick,myolive,rounded corners](2,0)--(2,1.25)--(1,1.75)--(1,2.25)--(0,2.75)--(0,3);
\draw [thick,myblue,rounded corners](2,1.25)--(2.3,1.5);
\draw [thick,myred,rounded corners](2.3,1.5)--(2,1.75);
\draw [thick,myblue,rounded corners](1,0.25)--(1.3,0.5);
\draw [thick,myred,rounded corners](1.3,0.5)--(1,0.75);
\draw [thick,myblue,rounded corners](1,2.25)--(1.3,2.5);
\draw [thick,myred,rounded corners](1.3,2.5)--(1,2.75);
\draw (2.3,1.5) node[right]{{\tiny $\pi$}}; \draw (2,1.2) node[right]{{\tiny $C$}};
\draw (2,1.8) node[right]{{\tiny $A$}};
\draw (1.3,2.5) node[right]{{\tiny $\pi$}}; \draw (1,2.2) node[right]{{\tiny $C$}};
\draw (1,2.8) node[right]{{\tiny $A$}};
\fill [myviolet] (1.3,0.5) circle (0.1);
\fill [myviolet] (2.3,1.5) circle (0.1);
\fill [myviolet] (1.3,2.5) circle (0.1);
\node  at (3.2,1.5){\Large $\stackrel{?}{=}$};
\end{tikzpicture}
\begin{tikzpicture}[scale=1]
\draw (0,0) node[below]{$M_i$}; \draw (2,0) node[below]{$M_j$};
\draw (1,0) node[below]{$M_k$};
\draw [ultra thick,myolive,  rounded corners](1,0)--(1,0.25)--(2,0.75)--(2,2.25)--(1,2.75)--(1,3);
\draw [ultra thick,myolive,  rounded corners](0,0)--(0,1.25)--(1,1.75)--(1,2.25)--(2,2.75)--(2,3);
\draw [ultra thick,myolive,  rounded corners](2,0)--(2,0.25)--(1,0.75)--(1,1.25)--(0,1.75)--(0,3);
\draw (2.3,0.5) node[right]{{\tiny $\pi$}}; \draw (2,0.2) node[right]{{\tiny $C$}};
\draw (2,0.8) node[right]{{\tiny $A$}}; 
\draw (2.3,2.5) node[right]{{\tiny $\pi$}}; \draw (2,2.2) node[right]{{\tiny $C$}};
\draw (2,2.8) node[right]{{\tiny $A$}};
\draw (1.3,1.5) node[right]{{\tiny $\pi$}}; \draw (1,1.2) node[right]{{\tiny $C$}};
\draw (1,1.8) node[right]{{\tiny $A$}};
\draw [thick,myblue,rounded corners](2,0.25)--(2.3,0.5);
\draw [thick,myred,rounded corners]((2.3,0.5)--(2,0.75);
\draw [thick,myblue,rounded corners](1,1.25)--(1.3,1.5);
\draw [thick,myred,rounded corners](1.3,1.5)--(1,1.75);
\draw [thick,myblue,rounded corners](2,2.25)--(2.3,2.5);
\draw [thick,myred,rounded corners](2.3,2.5)--(2,2.75);
\fill [myviolet] (2.3,0.5) circle (0.1);
\fill [myviolet] (1.3,1.5) circle (0.1);
\fill [myviolet] (2.3,2.5) circle (0.1);
\end{tikzpicture}

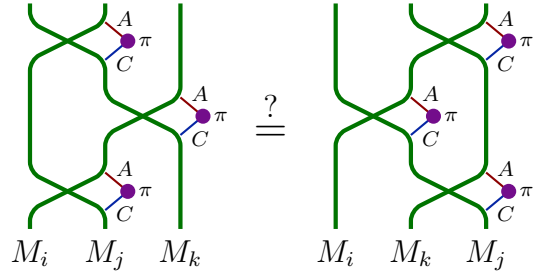
\captionof{figure}{Claim of the theorem}\label{pic:proof}
\end{center}
We work on both sides of the equality in order to save space. 
Using the naturality and the symmetry of~$c$, one moves all the blue-red coaction-action circuits to the bottom of the diagrams. Fig.~\ref{pic:proof2} contains the resulting picture.
\begin{center}
\begin{tikzpicture}[scale=1]
\draw [ultra thick,myolive,  rounded corners](1,0)--(1,1.25)--(0,1.75)--(0,2.25)--(1,2.75)--(1,3);
\draw [ultra thick,myolive,  rounded corners](2,0)--(2,1.75)--(1,2.25)--(0,2.75)--(0,3);
\draw [ultra thick,myolive,  rounded corners](0,0)--(0,1.25)--(1,1.75)--(2,2.25)--(2,3);
\draw [thick,myblue,rounded corners](2,0.35)--(2.3,0.475);
\draw [thick,myred,rounded corners]((2.3,0.475)--(1,0.6);
\draw [thick,myblue,rounded corners](1,0.25)--(1.3,0.375)--(1.6,0.725);
\draw [thick,myred,rounded corners](1.6,0.725)--(1.3,0.825)--(0,0.9);
\draw [thick,myblue,rounded corners](2,0.15)--(2.3,0.275)--(2.6,0.575);
\draw [thick,myred,rounded corners](2.6,0.575)--(2.3,0.975)--(0,1.1);
\fill [myviolet] (2.6,0.575) circle (0.08);
\fill [myviolet] (2.3,0.475) circle (0.08);
\fill [myviolet] (1.6,0.725) circle (0.08);
\node  at (3.5,1.5){\Large $\stackrel{?}{=}$};
\end{tikzpicture}
\begin{tikzpicture}[scale=1]
\node  at (-0.5,1.5){$ $};
\draw [ultra thick,myolive,  rounded corners](1,0)--(1,1.25)--(2,1.75)--(2,2.25)--(1,2.75)--(1,3);
\draw [ultra thick,myolive,  rounded corners](0,0)--(0,1.75)--(1,2.25)--(2,2.75)--(2,3);
\draw [ultra thick,myolive,  rounded corners](2,0)--(2,1.25)--(1,1.75)--(0,2.25)--(0,3);
\draw [thick,myblue,rounded corners](2,0.15)--(2.3,0.275)--(2.6,0.425);
\draw [thick,myred,rounded corners](2.6,0.425)--(2.3,0.625)--(1,0.75);
\draw [thick,myblue,rounded corners](1,0.85)--(1.3,0.975);
\draw [thick,myred,rounded corners](1.3,0.975)--(0,1.1);
\draw [thick,myblue,rounded corners](2,0.35)--(2.3,0.475);
\draw [thick,myred,rounded corners](2.3,0.475)--(0,0.6);
\fill [myviolet] (2.3,0.475) circle (0.08);
\fill [myviolet] (2.6,0.425) circle (0.08);
\fill [myviolet] (1.3,0.975) circle (0.08);
\end{tikzpicture}

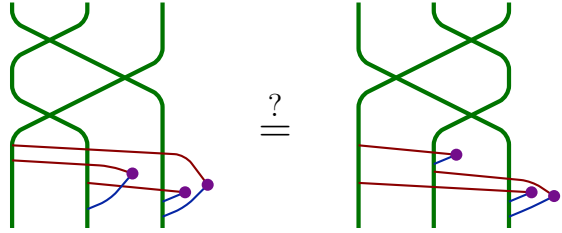
\captionof{figure}{Dragging down the coaction-action circuits}\label{pic:proof2}
\end{center}

The YBE for the braiding~$c$ allows one to identify the upper parts of the diagrams. To compare the lower parts, recall that ``coacting'' blue strands or ``acting'' red strands can be twisted near a thick green strand any number of times thanks to the defining property~\eqref{eqn:BrMod} of braided (co-)modules (Fig.~\ref{pic:cYBE}); we display these multiple twists by solid boxes. Moreover, acting and coacting strands can be switched using the defining property~\eqref{eqn:gen_YD} of generalized YD modules (Fig.~\ref{pic:gen_YD}); we apply this to the middle strand of the diagram on the right. The lower parts of our diagrams can thus be transformed to the ones on Fig.~\ref{pic:proof3}. 
\begin{center}
\begin{tikzpicture}[scale=1]
\draw [ultra thick,myolive,rounded corners](0,0)--(0,3) ;
\draw [ultra thick,myolive,rounded corners](1,0)--(1,3);
\draw [ultra thick,myolive,rounded corners](2,0)--(2,3);
\draw [thick,myblue,rounded corners](2,0.4)--(2.3,0.4)--(2.3,0.5);
\draw [thick,myred,rounded corners](2.5,1.7)--(0,2.2);
\draw [thick,myblue,rounded corners](1,0.25)--(1.5,0.25)--(1.5,1.6);
\draw [thick,myred,rounded corners](1.5,1.6)--(0,1.9);
\draw [thick,myblue,rounded corners](2,0.3)--(2.5,0.3)--(2.5,1.7);
\draw [thick,myred,rounded corners](2.3,1)--(1,1.3);
\fill [myred] (0.3,1.7) rectangle (0.5,2.3);\draw (0.4,1.7) node[below]{$\alpha_2$};
\fill [myblue] (2.2,0.5) rectangle (2.6,0.7);\draw (2.5,0.5) node[right]{$\gamma_2$};
\draw [thick,myblue,rounded corners](2.3,0.7)--(2.3,1);
\draw [thick,myblue,rounded corners](2.5,0.7)--(2.5,1.5);
\fill [myviolet] (1.5,1.6) circle (0.08);
\fill [myviolet] (2.5,1.7) circle (0.08);
\fill [myviolet] (2.3,1) circle (0.08);
\node  at (4,1.5){\Large $\stackrel{?}{=}$};
\node  at (5,1.5){$ $};
\end{tikzpicture}
\begin{tikzpicture}[scale=1]
\draw [ultra thick,myolive,rounded corners](0,0.2)--(0,3.2);
\draw [ultra thick,myolive,rounded corners](1,0.2)--(1,3.2);
\draw [ultra thick,myolive,rounded corners](2,0.2)--(2,3.2);
\draw [thick,myblue,rounded corners](2,0.3)--(2.5,0.3)--(2.5,0.7)--(2.5,1);
\draw [thick,myblue,rounded corners](2,0.4)--(2.3,0.4)--(2.3,0.8)--(1.3,1.3)--(1.3,1.5);
\fill [myblue] (2.2,0.5) rectangle (2.6,0.7) ;\draw (2.5,0.5) node[right]{$\gamma_1$};
\draw [thick,myred,rounded corners](2.5,1)--(1,1.9);
\draw [thick,myblue,rounded corners](1,0.55)--(1.7,0.8)--(1.7,1.35);
\draw [thick,myblue,rounded corners](1.7,1.6)--(1.7,1.9)--(1.3,2.2);
\draw [thick,myred,rounded corners](1.3,2.2)--(0,3);
\draw [thick,myred,rounded corners](1.3,1.5)--(1.3,1.9)--(0,2.7);
\fill [myred] (0.3,2.3) rectangle (0.5,2.9) ;\draw (0.5,2.3) node[below]{$\alpha_1$};
\fill [myviolet] (1.3,2.2) circle (0.08);
\fill [myviolet] (2.5,1) circle (0.08);
\fill [myviolet] (1.3,1.5) circle (0.08);
\end{tikzpicture}

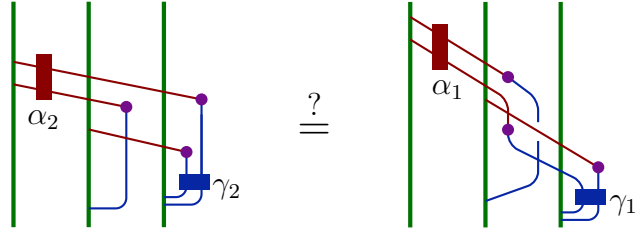
\captionof{figure}{Transformed lower parts of the diagrams}\label{pic:proof3}
\end{center}
This last assertion follows from our technical hypothesis~\eqref{eqn:cond_pi} (Fig.~\ref{pic:cond_pi}). 
\qedhere
\end{proof}

To illustrate the unifying nature of the notion of generalized YD modules and of the braidings provided by the theorem, we now interpret usual YD modules and representations of a crossed module of groups as YD modules over carefully chosen braided systems. The braidings given by our theorem in these two settings recover the usual braidings for these structures, recalled in the Introduction. Original examples will be treated in the following sections.

\begin{example}[Generalized YD modules generalize usual YD modules]\label{ex:YD'}
$ $

To a Hopf algebra $(H,\mu,\nu,\varepsilon,\Delta,S)$ in a symmetric monoidal category $(\C,\otimes,\II,c)$, one can associate the following rank $2$ braided system in~$\C$:
\begin{itemize}
\item its components are two copies of~$H$: $C=A=H$;
\item the braiding is defined by 
\begin{align*}
&\sigma_{C,C} = \sigma_{coAss}, \qquad \sigma_{A,A} = \sigma_{Ass},\\
&\sigma_{C,A} = (H \otimes \mu^2) \circ (c_{H,H} \otimes {H \otimes H}) \circ (S \otimes c_{H,H} \otimes H) \circ \\
& \qquad\qquad (c_{H,H} \otimes {H \otimes H}) \circ (H \otimes \Delta^2)
\end{align*} 
(where $\mu^2=\mu \circ (\mu \otimes H)$, $ \Delta^2 = (\Delta \otimes H) \circ \Delta$).
\end{itemize}
In~$\Vect$, the morphism~$\sigma_{C,A}$ takes the familiar form
\begin{align*}
\sigma_{C,A}(h \otimes h')& = h'_{(2)} \otimes S(h'_{(1)})hh'_{(3)}.
\end{align*}
Note that it is not a braiding on~$H$ in general. As mentioned in Examples~\ref{ex:Ass} and~\ref{ex:coAss}, the cYBE on $C\otimes C \otimes C$ and on $A\otimes A \otimes A$ follows from the coassociativity of~$\Delta$ and, respectively, from the associativity of~$\mu$. The verification of the remaining instances of the cYBE (on $C\otimes C \otimes A$ and on $C\otimes A \otimes A$) is lengthy but straightforward. 

Recall the braided module analysis from Examples~\ref{ex:Ass} and~\ref{ex:coAss}. Together with a comparison of the definition of~$\sigma_{C,A}$ with the defining relation~\eqref{eqn:YD} for usual YD modules, this identifies $\YD_A^C$ as the category $\YDnn_H^H$of non-normalized (in the sense of Examples~\ref{ex:Ass} and~\ref{ex:coAss}) YD modules over~$H$ in~$\C$. The category $\YDn_H^H = \YD_H^H$  of usual, or normalized, YD modules is its full subcategory:
$$\YD_H^H = \YDn_H^H \hookrightarrow \YDnn_H^H =\YD_A^C.$$

Now, consider the morphism $\pi = \Id_H$. A direct verification shows that it satisfies condition~\eqref{eqn:cond_pi} with $\alpha_1 = \alpha_2 = \gamma_1 = \gamma_2 = 1$. Thus Theorem~\ref{thm:genYD_br} applies here. For $\C = \Vect$, the braiding obtained is precisely the braiding~$\sigma_{YD}$ from~\eqref{eqn:br_YD}.

More generally, an analogous rank~$2$ braided system can be constructed for an $H$-bimodule coalgebra~$C$ and an $H$-bicomodule algebra~$A$. YD modules over this system yield a non-normalized version of $(H,A,C)$-crossed $H$-modules, as defined by S.~Caenepeel, G.~Militaru, and S.~Zhu~\cite{CaMiZhu}. These modules can thus be endowed with braidings, provided that additionally one has a connecting morphism $\pi \colon C \to A$.
\end{example}

\begin{example}[Representations of a crossed module of groups as generalized YD modules]\label{ex:CrMod'}
$ $

To a crossed module of groups $(K, G, \pi, \cdot)$ (see the Introduction), one can associate the following rank $2$ braided system in~$\Vect$ (or analogously in $\ModCat_{R}$ for a unital commutative ring~$R$):
\begin{itemize}
\item as components, take $C=\kk K$ and $A=\kk G$;
\item the braiding is defined by
\begin{align*}
&\sigma_{C,C} = \sigma_{coAss}, \qquad \sigma_{A,A} = \sigma_{Ass},\\
&\sigma_{C,A} (k \otimes g) = g \otimes (k \cdot g),
\end{align*} 
where~$\kk K$ and~$\kk G$ are endowed with the usual Hopf algebras structure (given by a linearization of the maps $\Delta(g) = g \otimes g$, $\varepsilon(g)=1$, $\mu(g \otimes g') = gg'$, $\nu(1)=e$, $S(g) = g^{-1}$).
\end{itemize}
As usual, the cYBE on $C\otimes C \otimes C$ and on $A\otimes A \otimes A$ follows from the coassociativity of~$\Delta$ and, respectively, from the associativity of~$\mu$. The cYBE on $C\otimes C \otimes A$ is obvious, and on $C\otimes A \otimes A$ it is a consequence of~$\cdot$ being a group action. 

Now, consider a $\kk K$-coaction~$\delta$ on~$M$ which is counital, in the sense of $(M \otimes \varepsilon) \circ \delta = M$, where $\varepsilon(k) = 1 \in \kk$ for all $k \in K$. Such a coaction is the same thing as a $K$-grading: the correspondence is given by
\begin{align*}
\delta(m) &= \sum\nolimits_{k \in K} m_k \otimes k, & \text{and }\qquad m_k &= (M \otimes \partial_k)\circ \delta(m),
\end{align*}  
where~$\partial_k$ is the linearization of the \emph{Kronecker delta map} $\partial_k(k') = \delta_{k,k'}$. Moreover, a map $M \to M'$ is compatible with $\kk K$-coactions if and only if it preserves the corresponding $K$-gradings. Further, condition~\eqref{eqn:gen_YD} defining a generalized YD module is equivalent here to all $g \in G$ sending~$M_k$ onto~$M_{k \cdot g}$. Thus Bantay's representation category $\MMM(K,G)$ is the full subcategory of $\YD_{\kk G}^{\kk K}$ consisting of all normalized modules (in the sense of Examples~\ref{ex:Ass} and~\ref{ex:coAss}). In other words, one has a category inclusion
$$\MMM(K,G) = \YDn_{\kk G}^{\kk K} \hookrightarrow \YDnn_{\kk G}^{\kk K} = \YD_{\kk G}^{\kk K}.$$

Condition~\eqref{eqn:cond_pi} for~$\pi$ holds true with $\alpha_1 = \alpha_2 = \gamma_1 = \gamma_2 = 1$; this follows from~\eqref{eqn:CrModGr} and from~$\pi$ being a group morphism. The braiding from Theorem~\ref{thm:genYD_br} coincides here with~$\sigma_{CrMod}$ from~\eqref{eqn:br_CrMod}.
\end{example}

\begin{remark}[Generalized YD modules as braided modules]
$ $

If the object~$C$ admits a dual~$C^*$ in~$\C$, then the braided system $(C,A;\osigma)$ can be partially dualized to $(A,C^*;\osigma^*)$, with a category isomorphism
\begin{align}\label{eqn:gYD_dual}
\YD_A^C &\simeq \ModCat_{(A,C^*;\osigma^*)}.
\end{align}
In the context of Example~\ref{ex:YD'}, the dual~$C^*$ exists if~$H$ is a finite-dimensional Hopf algebra over~$\kk$ (or at least is graded and of finite dimension in every degree). In this case \eqref{eqn:gYD_dual} can be continued as 
\begin{align}\label{eqn:YD_dual}
\YD_H^H = \YDn_H^H &\hookrightarrow \YDnn_H^H \simeq \ModCat_{(H,H^*;\osigma^*)} \simeq \Modnn_{\Dr (H)},
\end{align}
where the algebra~$\Dr(H)$ is the Drinfel$'$d double of~$H$. 
In Example~\ref{ex:CrMod'}, the dual exists if~$K$ is a finite group, in which case~\eqref{eqn:gYD_dual} can be continued as 
\begin{align}\label{eqn:CrMod_dual}
\MMM(K,G) &\hookrightarrow \YDnn^{\kk K}_{\kk G} \simeq \ModCat_{(\kk G,(\kk K)^*;\osigma^*)} \simeq \Modnn_{(\kk K)^* \rtimes \kk G}.
\end{align}
Explicitly, $(\kk K)^*$ has a standard $\kk$-linear basis given by the delta maps~$\partial_k$, $k \in K$, which form a complete orthogonal system of idempotents, and~$G$ acts on $(\kk K)^*$ by algebra automorphisms according to the rule 
$$g \cdot \partial_k = \partial_{k \cdot g^{-1}}.$$ 
Note that originally Bantay \textit{defined}~$\MMM(K,G)$ as the category $\Modn_{(\kk K)^* \rtimes \kk G}$. See~\cite{Lebed,Lebed2,Lebed2ter} for a general treatment of the situations in which category isomorphisms of type \eqref{eqn:YD_dual}-\eqref{eqn:CrMod_dual} emerge.
\end{remark}

\section{Representations of a crossed module of shelves}\label{sec:RepsCrModSh}

In~\cite{CransWag}, A.~Crans and the second author generalized the notion of crossed module of groups to that of crossed module of racks, and studied its properties. We now recall their definition, and extend it to the case of shelves (see Source 2 in the Introduction for definitions). We further propose a notion of representations of such crossed modules, include it into the framework of generalized YD modules, and, using Theorem~\ref{thm:genYD_br}, obtain a new source of braidings. This source comprises both the self-distributivity braiding~$\sigma_{SD}$ and Bantay's braiding~$\sigma_{CrMod}$.   

\begin{definition}
A \emph{rack} is a shelf $(R,\op)$ for which all the \emph{right translations} 
$$t_r \colon \wr \mapsto \wr \op r$$ 
are bijective; their inverses are denoted by $\wr \mapsto \wr \wop r$.
\end{definition}

\begin{example}
A group~$G$ with the conjugation operation $g \op g' = (g')^{-1} gg'$ is a rack, called \emph{the conjugation rack of~$G$}, and denoted by $Conj(G)$.
\end{example}

\begin{definition}[Crossed modules of shelves and racks]\label{def:CrModSh}
$ $

\begin{itemize}
\item A \emph{shelf/rack morphism} between shelves/racks $(R,\op)$ and $(S,\op)$ is a map $f \colon R \to S$ intertwining their shelf operations:
$$f(r \op \wr) = f(r) \op f(\wr).$$
Shelf/rack iso-, endo- and automorphisms are defined analogously.
\item Given a shelf $(S, \op)$, an \emph{$S$-set} is a set~$M$ with a map $\mop \colon M \times S \to M$ (sometimes seen as a map $\varphi \colon S \to \End_{\Set}(M)$) satisfying
\begin{align}
(m \mop s) \mop \ws &= (m \mop \ws) \mop (s \op \ws), &&m \in M,\, s,\ws \in S. \label{eqn:ShAction}
\end{align}
\item Given a rack $(R, \op)$, an \emph{$R$-rack-set} is an $R$-set~$M$ on which~$R$ acts by bijections, i.e., the maps $m \mapsto m \mop r$ are invertible for all~$r \in R$.
\item The maps~$\mop$ above are called \emph{shelf/rack actions}.
\item \emph{$S$-modules} (or \emph{$R$-rack-modules}) in an arbitrary category~$\C$ are defined as maps $\varphi \colon S \to \End_{\C}(M)$ (or $\varphi \colon R \to \Aut_{\Set}(M)$) satisfying $\varphi(s) \varphi(\ws) = \varphi(\ws) \varphi(s \op \ws)$.
\item A \emph{crossed module of shelves} is the data of a shelf morphism $\pi \colon R \to S$ and a shelf action~$\cdot$ of~$S$ on~$R$ by shelf morphisms, compatible in the sense of
\begin{align}
r \cdot \pi(\wr) &= r \op \wr, &&r,\wr \in R, \label{eqn:CrModSh}\\
\pi (r \cdot s) &= \pi(r) \op s, &&r \in R,\, s \in S.\label{eqn:CrModSh'}
\end{align}
\item If~$R$ and~$S$ above are racks, with~$S$ acting by rack {auto}morphisms, one talks about a \emph{crossed module of racks}. 
\item An \emph{augmented rack} is the data of a group~$G$, a $G$-set~$R$, and a $G$-equivariant map $\pi \colon R \to G$, in the sense of~\eqref{eqn:CrModGr'}.
\end{itemize}
\end{definition}

\begin{remark}[An alternative definition]\label{rmk:GenAugShelf}
$ $

The definition of a crossed module of shelves/racks is redundant: it suffices to have a \emph{generalized augmented shelf/rack}, that is, a shelf/rack $S$, an $S$-set or $S$-rack-set $R$, and an $S$-equivariant map $\pi \colon R \to S$ (in the sense of~\eqref{eqn:CrModSh'}). For this data, relation~\eqref{eqn:CrModSh} can be taken as the definition of a shelf/rack operation on~$R$, called \emph{the induced operation}; with this choice, $\pi$ becomes a shelf morphism, and~$S$ acts on~$R$ by shelf (auto)morphisms. See~\cite{CransWag} for more details. We keep the original definition in order to better see its analogy with that of a crossed module of groups, but in practice often turn to the lighter one. 
\end{remark}

\begin{remark}[An augmented rack as a crossed module of racks]\label{rmk:Induced}
$ $

Take an augmented rack $(R,G,\pi,\cdot)$. Since a group action by~$G$ can be viewed as a rack action by $Conj(G)$, one obtains a generalized augmented rack $(R,Conj(G),\pi,\cdot)$, which is in fact a crossed module of racks (Remark~\ref{rmk:GenAugShelf}). In particular, $R$ can be endowed with the induced rack structure $r \op \wr = r \cdot \pi(\wr)$, which justifies the term \emph{augmented rack}.
\end{remark}

\begin{example}[Augmentation over the associated group of a rack]\label{ex:AssGroup}
$ $

The \emph{associated group} $Ass(R)$ of a rack $(R, \op)$ is the free group on~$R$ modulo the relations 
\begin{align}\label{eqn:AssGroup}
r \wr &= \wr (r \op \wr), & r,\wr &\in R.
\end{align}
This construction actually defines a functor $\operatorname{Ass}$ from the category of racks to that of groups; its right adjoint $\operatorname{Conj}$ stems from the conjugation rack construction. The group $Ass(R)$ acts on~$R$ via 
\begin{align}\label{eqn:RActAssRAct}
r \cdot \wr &= r \op \wr, &r \cdot (\wr)^{-1} &= r \wop \wr.
\end{align} 
The tautological map $\pi_{Ass} \colon R \to Ass(R)$, $r \mapsto r$ is $Ass(R)$-equivariant. Thus every rack can be augmented over~$Ass(R)$. This augmentation is \emph{universal}, in the sense that for any augmented rack structure $(R,G,\pi,\cdot)$ with the same~$R$, the map~$\pi \colon R \to G$ factors through~$\pi_{Ass}$. Note also that~$\pi_{Ass}$ induces a bijection between $Ass(R)$-modules and $R$-rack-modules in any category (its inverse is given by formulas analogous to~\eqref{eqn:RActAssRAct}). Our last observation concerns the case when $(R, \op)$ is simply a shelf: $R$ is then acted on by its \emph{associated monoid} only, i.e, the free monoid on~$R$ modulo~\eqref{eqn:AssGroup}.
\end{example}

\begin{example}[Augmentation over the automorphism group of a rack]\label{ex:AutGroup}
$ $

Another augmentation of a rack $(R, \op)$ is given by the map $\pi \colon R \to \Aut(R)$ sending an $r \in R$ to the right translation map $t_r$, which is indeed a rack automorphism of~$R$. By definition, $R$ carries an $\Aut(R)$-action. The map $\pi_{Aut} \colon R \to \Aut(R)$, $r \mapsto t_r$ is easily shown to be $\Aut(R)$-equivariant, completing our augmented structure.
\end{example}

Note that in both examples above, the induced operation on~$R$ is in fact its original rack operation. According to Remark~\ref{rmk:Induced}, one thus obtains crossed modules of racks with an arbitrary rack as the ``$R$-part'' of the structure. 

We now mimic the development of the representation theory of a crossed module of groups in the generalized setting of a crossed module of shelves. 

\begin{definition}[Representations of a crossed module of shelves/racks]\label{def:CrModShRep}
$ $

\begin{itemize}
\item A \emph{set-theoretic / linear representation of a crossed module of shelves} $(R,S,\pi,\cdot)$ is an $S$-module $(M,\mop)$ in~$\Set$ / in~$\Vect$, endowed with an $R$-grading satisfying the compatibility condition 
$$M_r \mop s \subseteq M_{r \cdot s}.$$ The category of such representations (with the obvious notion of morphisms) is denoted by $\MMM_{\Set}(R,S,\pi,\cdot)$, or simply $\MMM_{\Set}(R,S)$ when this does not lead to confusion. 
The notation $\MMM_{\kk}(\ldots)$ is used in the linear setting.
\item If $(R,S,\pi,\cdot)$ above is a crossed module of racks and $(M,\mop)$ is an $S$-rack-module, then we talk about \emph{representations of a crossed module of racks} and use the notations $\MMM^{\operatorname{R}}_{\bullet}(\ldots)$.
\end{itemize}
\end{definition}

Note that a representation of a crossed module of racks satisfies a stronger compatibility condition $M_r \mop s = M_{r \cdot s}$.

\begin{example}[Adjoint representations]\label{ex:AdjRep}
$ $

Given a shelf/rack $(S,\op)$, the map $\pi = \Id_{S} \colon S \to S$ together with $s \cdot s' = s \op s'$ define a crossed module of shelves/racks, for which~$S$ itself is a representation (called \textit{adjoint}), with $s \mop s' = s \op s'$ and $S_s = \{s\}$. 
\end{example}

\begin{example}[A crossed module of groups as a crossed module of racks]\label{ex:CrModGrAsR}

$\quad$ A crossed module of groups $(K,G,\pi,\cdot)$ is in particular an augmented rack, and thus (Remark~\ref{rmk:Induced}) gives rise to the crossed modules of racks $(K,Conj(G),\pi,\cdot)$, with the induced rack operation $k \op \wk = k \cdot \pi(\wk)$ on~$K$. Relation~\eqref{eqn:CrModGr} transforms it into $k \op \wk = (\wk)^{-1} k \wk$, so our crossed modules of racks can be written as $(Conj(K),Conj(G),\pi,\cdot)$. Observe the tautological inclusion of the set-theoretic / linear representation categories
\begin{align}\label{eqn:CrModGrAsR}
& \MMM_{\bullet}(K,G) \hookrightarrow \MMM^{\operatorname{R}}_{\bullet}(Conj(K),Conj(G)).
\end{align}
It is in general strict. Indeed, taking as~$K$ the trivial group, one identifies $\MMM_{\bullet}(\{1\},G)$ with the category of $G$-modules, and $\MMM^{\operatorname{R}}_{\bullet}(\{1\},Conj(G))$ with the category of $Conj(G)$-rack-modules. Now, take a $G$-module $(M, \ast)$ with an inversion $m \mapsto \overline{m}$ satisfying $\overline{m} \ast g = \overline{m \ast g}$ (e.g., the map $m \mapsto -m$ in the linear setting). The operation $m \mop g = \overline{m} \ast g$ defines a $Conj(G)$-rack-module structure on~$M$ which is not necessarily a $G$-module structure.
\end{example}

\begin{example}[A crossed module of racks as a crossed module of groups]\label{ex:CrModRAsGr}

$\quad$ A crossed module of racks $(R,S,\pi,\cdot)$ induces a crossed modules structure $(Ass(R),Ass(S),\widetilde{\pi},\widetilde{\cdot})$ for the associated groups. An $S$-rack-module structure on~$M$ is equivalent to an $Ass(S)$-module structure on~$M$, and an $R$-grading $M = \oplus_{r \in R}M_r$ induces an $Ass(R)$-grading:  put $M_{\pi_{Ass}(r)} = \displaystyle{\oplus_{\wr\, | \, \pi_{Ass}(\wr) = \pi_{Ass}(r)}} M_{\wr}$, and declare $M_r$ trivial for $r$ outside $\pi_{Ass}(R)$. Analyzing the compatibility conditions, one sees that this yields a map 
\begin{align}\label{eqn:CrModRAsGr}
& \MMM^{\operatorname{R}}_{\bullet}(R,S) \to \MMM_{\bullet}(Ass(R),Ass(S))
\end{align}
between the corresponding set-theoretic / linear representation categories. This map is neither injective nor surjective in general. Indeed, for the cyclic rack $R_{cycl}= (\ZZ, \, r \op_c \wr = r+1)$, the associated group is the free group~$\langle t \rangle$ on one element (since $rr = r(r \op_c r)$ implies here $r = r+1$ for all~$r$). The representations of the crossed module of racks $(R_{cycl},R_{cycl},\Id_{\ZZ},\op_c)$ (cf. Example~\ref{ex:AdjRep}) are $\ZZ$-graded sets / vector spaces~$M$ endowed with a bijection $f \colon M \to M$ such that $f(M_r) = M_{r+1}$ (the $R_{cycl}$-action being defined by $m \mop r = f(m)$ for all~$r$), whereas the representations of the associated crossed module of groups $(\langle t \rangle,\, \langle t \rangle,\, \Id_{\langle t \rangle},\, g \op g' = g)$ are $\langle t \rangle$-graded~$M$ endowed with a bijection $f \colon M \to M$ preserving the grading (and inducing the $\langle t \rangle$-action $m \ast t^\alpha = f^\alpha(m)$). The correspondence~\eqref{eqn:CrModRAsGr} sends a representation $(M,\, f,\, gr)$ to $(M,\, f,\, gr_0 \colon m \mapsto t)$. On the one hand, it totally forgets the grading~$gr$ and is thus not injective; on the other hand, in its image everyone lives in degree~$t$, hence the non-surjectivity.
\end{example}

\begin{example}[Crossed modules of racks versus crossed modules of shelves]
$\quad$ A crossed module of racks is in particular a crossed module of shelves, thus accepting two types of representation theories, corresponding to the categories $\MMM^{\operatorname{R}}_{\bullet}$ and $\MMM_{\bullet}$. The second category is strictly larger in general. Indeed, one can transform a representation $(M= \oplus_{r \in R}M_r, \mop) \in \MMM_{\bullet}(R,S)$ into the following one:
$$(M \oplus M, \, (m \oplus \wm) \mop s = (m \mop s + \wm \mop s) \oplus 0, \, (M \oplus M)_r =  M_r \oplus M_r)$$ 
(with the obvious modifications in the set-theoretic setting). It does not belong to $\MMM^{\operatorname{R}}_{\bullet}(R,S)$, since the action of any $s \in S$ is non-invertible.
\end{example}

\begin{proposition}[Crossed modules of shelves as braided systems]\label{prop:CrModShBrSyst}
$ $

For a crossed module of shelves $(R,S,\pi,\cdot)$, the following data define a rank $2$ braided system in~$\Set$:
\begin{itemize}
\item as components, take $C=R$ and $A=S$;
\item the braiding is defined by
\begin{align*}
\sigma_{C,C} = \sigma_{coAss} \colon & r \otimes r' \mapsto r' \otimes r', \\
\sigma_{A,A} = \sigma_{SD} \colon & s \otimes s' \to s' \otimes (s \op s'),\\
\sigma_{C,A} \colon & r \otimes s \mapsto s \otimes (r \cdot s).
\end{align*} 
\end{itemize}
By linearization, this braided system can be transformed into one in~$\Vect$. 
\end{proposition}

\begin{proof}
The cYBE on $C\otimes C \otimes C$ and $A\otimes A \otimes A$ are taken care of by Examples~\ref{ex:coAss} and~\ref{ex:SDbr}. The cYBE on $C\otimes C \otimes A$ is obvious, and on $C\otimes A \otimes A$ it follows from the fact that~$\cdot$ is an $S$-action.
\end{proof}

\begin{remark}
In fact the component~$\sigma_{C,C}$ of the braiding above can also be seen as a self-distributivity  braiding, by considering the shelf operation $r \op_0 r' = r'$ on~$R$. Even better: the shelf operations~$\op$ on~$S$ and~$\op_0$ on~$R$ can be extended to a shelf operation~$\op$ on $T =S \sqcup R \sqcup \{e\}$ by putting $r \op s = r \cdot s$, $s \op r = e$, $e \op t = e$, and $t \op e = t$ for all $s \in S, r \in R, t \in T$. All the braiding components from the proposition are now particular cases of the braiding~$\sigma_{SD}$ on $(T,\op)$. 
\end{remark}

\begin{remark}\label{rmk:CrModShBrSyst}
Another rank~$2$ braided system in~$\Set$ can be defined by the same data as in the proposition except for~$\sigma_{C,C}$, which becomes 
\begin{align*}
&\sigma_{C,C} = \sigma_{SD} \colon r \otimes r' \to r' \otimes (r \op r').
\end{align*} 
 The instances of the cYBE changed with respect to the previous structure are those on $C\otimes C \otimes C$, which is an application of Example~\ref{ex:SDbr}, and on $C\otimes C \otimes A$, where it follows from the fact that~$S$ acts on~$R$ by shelf morphisms. Once again, all the braided components of this system can be seen as parts of a single self-distributivity braiding on $(T =S \sqcup R, \op)$, where $\op$ extends the shelf operations on~$S$ and~$R$ by $r \op s = r \cdot s$ and $s \op r = s \op \pi(r)$.
\end{remark}

\begin{proposition}[Representations of a crossed module of shelves as generalized YD modules]\label{prop:CrModShCatIso}
$ $

In the settings of the previous proposition, one has category inclusions
\begin{align*}
\MMM_{\Set}(R,S) & \hookrightarrow \YD_S^R, & \MMM_{\kk}(R,S)  & \hookrightarrow \YD_{\kk S}^{\kk R}.
\end{align*}
\end{proposition}

\begin{proof}
As recalled in Examples~\ref{ex:coAss} and~\ref{ex:SDbr}, an $S$-module is the same thing as a braided module over $(S;\sigma_{SD})$, and a comodule over $(R,\Delta \colon r \mapsto r \otimes r, \varepsilon \colon r \mapsto 1)$ is automatically a braided comodule over $(R;\sigma_{coAss})$. One then interprets an $R$-grading as the $R$-comodule structure 
\begin{align*}
m &\mapsto m \times gr(m), & \text{or } \qquad m &\mapsto \sum\nolimits_{r \in R} m_r \otimes r
\end{align*}
(depending on the context), and identifies the compatibility condition $M_r \mop s \subseteq M_{r \cdot s}$ with~\eqref{eqn:gen_YD} for our~$\sigma_{C,A}$.
\end{proof}

More precisely, these generalized YD modules can be viewed as decorated versions of the representations from $\MMM_{\bullet}(R,S)$:
\begin{proposition}[Twisted representations]\label{prop:CatYDSR}
$ $

Take a crossed module of shelves $(R,S,\pi,\cdot)$. 
\begin{enumerate}
\item The category $\YD_S^R$ is isomorphic to the category of set-theoretic representations $(M, \mop, gr)$ of $(R,S,\pi,\cdot)$ endowed with an additional map $f \colon M \to M$ which
\begin{itemize}
\item respects the $R$-grading~$gr$, and 
\item intertwines the $S$-action~$\mop$ (i.e., $f(m \mop s) = f(m) \mop s$). 
\end{itemize} 
\item The category $\YD_{\kk S}^{\kk R}$ is isomorphic to the category of $\kk$-linear $S$-modules $(M,\mop)$ with the following additional data:
\begin{enumerate}
\item a distinguished $S$-stable subspace~$M'$ with a compatible $R$-grading, in the sense of 
\begin{align*}
\sum\nolimits_{r' \, | \,r' \cdot s = r }m'_{r'} \mop s &\in M'_{r}, && m' \in M', r \in R, s \in S,
\end{align*}
\item
and a surjection $f \colon M \twoheadrightarrow M'$ which
\begin{itemize}
\item respects the $R$-grading when restricted to~$M'$, and 
\item intertwines the $S$-actions.
\end{itemize} 
\end{enumerate} 
\end{enumerate}
For both categories, morphisms are defined in the usual way.
\end{proposition}

The category inclusions from Proposition~\ref{prop:CrModShCatIso} are realized by taking $f = \Id_M$ or, respectively, $M' = M$ and $f = \Id_M$.

\begin{proof}
One follows the proof of Proposition~\ref{prop:CrModShCatIso}, treating braided comodules over $(R;\sigma_{coAss})$ with more care. One sees that the $R$-coaction has to be of the form
\begin{align*}
m &\mapsto f(m) \times gr(m), &\text{or } \qquad m &\mapsto \sum\nolimits_{r \in R} f(m)_r \otimes r.
\end{align*} 
The compatibility relation~\eqref{eqn:gen_YD} is then translated into a list of requirements for the map~$f$ and for the behavior of the $R$-grading under the $S$-action.
\end{proof}

\begin{definition}\label{def:twisted}
The categories described in the proposition are denoted by $\MMM^{\operatorname{tw}}_{\bullet}(R,S)$, or $\MMM^{\operatorname{R; tw}}_{\bullet}(R,S)$ in the rack case. Their objects are called \emph{twisted representations} of the corresponding crossed module of shelves/ racks, and the maps~$f$ are referred to as the \emph{twisting maps}. 
\end{definition}

Proposition~\ref{prop:CatYDSR} thus establishes category equivalences
\begin{align*}
\MMM^{\operatorname{tw}}_{\Set}(R,S) &\simeq \YD_S^R, & \MMM^{\operatorname{tw}}_{\kk}(R,S) &\simeq  \YD_{\kk S}^{\kk R}.
\end{align*} 
In what follows we will freely switch between the generalized YD and the twisted viewpoints.

Using the $S$-equivariance relation~\eqref{eqn:CrModSh'} for~$\pi$, one readily checks condition~\eqref{eqn:cond_pi} with $ \alpha_2 = \gamma_1 = \gamma_2 = 1$ and $\alpha_1 = 0$ (observe that $\sigma_{A,A}$ is in general not idempotent in this setting, and the choice $\alpha_1 = 1$ from the previous examples would not work; cf. Remark~\ref{rmk:idempot}). Theorem~\ref{thm:genYD_br} is thus applicable here, yielding
 
\begin{theorem}[Representations of a crossed module of shelves are braided]\label{thm:SrModShBr} 

$\quad$ Any representations $(M_i, \mop_i, gr_i)$
 of a crossed module of shelves $(R,S,\pi,\cdot)$ in~$\Set$ (where $gr_i \colon M_i \to R$ are the grading maps) form a braided system, with the braidings
\begin{align}\label{eqn:SigmaCrModSh}
\sigma_{CrModSh}(m \otimes n)& = n \otimes m \mop_i \pi(gr_j(n))
\end{align}
on $M_i \otimes M_j$. Similar braidings exist for representations in~$\Vect$.
\end{theorem}

\begin{example}\label{ex:CrModGrAsSh'}
In the settings of Example~\ref{ex:CrModGrAsR}, one recovers Bantay's braiding~$\sigma_{CrMod}$ for crossed modules of groups (see~\eqref{eqn:br_CrMod} for the definition).
\end{example}

\begin{example}\label{ex:CrModGrAsSh''}
In the settings of Example~\ref{ex:CrModRAsGr}, the braidings~$\sigma_{CrModSh}$ and~$\sigma_{CrMod}$ for a representation of a crossed module of racks considered in the categories $\MMM^{\operatorname{R}}_{\bullet}(R,S)$ and, respectively, $\MMM_{\bullet}(Ass(R),Ass(S))$ (via the functor~\eqref{eqn:CrModRAsGr}) coincide.
\end{example}

\begin{example}\label{ex:AdjRep'}
For a shelf~$(S,\op)$ seen as the adjoint representation of $(S, S,$ $\Id_S, \op)$ 
(Example~\ref{ex:AdjRep}), $\sigma_{gYD}$ is the usual self-distributivity braiding~$\sigma_{SD}$ from~\eqref{eqn:br_SD}. More generally, a crossed module of shelves $(R,S,\pi,\cdot)$ has a representation~$(R,\cdot,\Id_R)$, for which $\sigma_{gYD}$ recovers once again the self-distributivity braiding~$\sigma_{SD}$.
\end{example}

\begin{remark}\label{rmk:BrYDSR}
For twisted representations $(M_i, \mop_i, gr_i,f_i) \in \MMM^{\operatorname{tw}}_{\bullet}(R,S)$, Theorem~\ref{thm:genYD_br} yields the braidings
\begin{align*}
\sigma_{TwCrModSh}(m \otimes n)& = f_j(n) \otimes m \mop_i \pi(gr_j(n))
\end{align*}
on $M_i \otimes M_j$, and similar formulas in the linear setting. They can be regarded as the braidings~\eqref{eqn:SigmaCrModSh} with extra ``$f$-twists''.
\end{remark}

\section{Representations of a crossed module of Leibniz algebras}\label{sec:RepsCrModLA}

In this section we recall the notion of crossed module of Leibniz algebras (cf. Example~\ref{ex:Lei}) and interpret it in terms of a rank~$2$ braided system. YD modules over such a system are then natural candidates for being called representations of the corresponding crossed module. We describe them explicitly, and endow them with braidings, supplied as usual by Theorem~\ref{thm:genYD_br}. This yields a new source of braidings, comprising~$\sigma_{Lei}$ from Example~\ref{ex:Lei}. Here we work in $\Vect$ for simplicity, but everything remains valid in a general symmetric additive monoidal category.

\begin{definition}[Crossed modules of Leibniz algebras]\label{def:CrModLA}
$ $

\begin{itemize}
\item A \emph{unital Leibniz algebra morphism} between unital Leibniz algebras $(\k, \,[,]_{\k}\, , 1_{\k})$ and $(\g, \,[,]_{\g}\, , 1_{\g})$ is a linear map $f \colon \k \to \g$ intertwining their structures:
\begin{align*}
f([k,\wk]_{\k}) &= [f(k),f(\wk)]_{\g}, & f(1_{\k}) &= 1_{\g}.
\end{align*}
\item A \emph{derivation} of a unital Leibniz algebra $(\k, \,[,]_{\k}\, , 1_{\k})$ is a linear map $f \colon \k \to \k$ satisfying
\begin{align*}
f([k,\wk]_{\k}) &= [k,f(\wk)]_{\k} + [f(k),\wk]_{\k}, & f(1_{\k}) &= 0.
\end{align*}
\item A \emph{representation} of $(\k, \,[,]_{\k}\, , 1_{\k})$ is a vector space~$M$ together with a unital Leibniz algebra morphism $\varphi \colon \k \to \End_{\kk}(M)$ (cf. Example~\ref{ex:Lei} for the Leibniz structure on~$\End_{\kk}(M)$). One says that~$\k$ \emph{acts on~$M$}, and writes $m \cdot k = \varphi(k)(m)$.
\item All the definitions above admit obvious non-unital versions.
\item A \emph{crossed module of Leibniz algebras} is the data of a Leibniz algebra morphism $\pi \colon \k \to \g$ and a (right) $\g$-action~$\cdot$ on~$\k$ by derivations, compatible in the sense of
\begin{align}
k \cdot \pi(\wk) &= [k,\wk]_{\k}, & &k,\wk \in \k, \label{eqn:CrModLA}\\
\pi (k \cdot g) &= [\pi(k), g]_{\g}, & &k \in \k,\, g \in \g. \label{eqn:CrModLA'}
\end{align}
\end{itemize}
\end{definition}

The simplest examples of crossed modules of Leibniz algebras are:
\begin{itemize}
\item the identity map $\Id_{\k} \colon \k \to \k$ for a Leibniz algebra~$\k$, with the adjoint action $k \cdot \wk = [k,\wk]_{\k}$, and
\item the zero map $0\colon \k \to \g$ between an abelian Leibniz algebra~$\k$ (i.e., the bracket~$[,]_{\k}$ is zero) and an arbitrary Leibniz algebra~$\g$ acting on~$\k$. 
\end{itemize}

Our definition of crossed modules is an anti-symmetric version of the Loday-Pirashvili one~\cite{LoPi}: they make~$\g$ act on~$\k$ on the left and on the right, with additional compatibility conditions, whereas we restrict ourselves to trivial left actions. 

\begin{remark}[An alternative definition]\label{rmk:AltCrModLA}
$ $

Similarly to the case of shelves, the definition of a crossed module of Leibniz algebras is redundant: it suffices to have a Leibniz algebra~$\g$ acting on a vector space~$\k$, and a $\g$-equivariant map $\pi \colon \k \to \g$ (in the sense of~\eqref{eqn:CrModLA'}). Relation~\eqref{eqn:CrModLA} then defines a Leibniz structure on~$\k$, on which $\g$ acts by derivations, and~$\pi$ becomes a Leibniz algebra morphism.
\end{remark}

It is natural to ask how to define crossed modules for unital Leibniz algebras. The naive definition does not work: condition~\eqref{eqn:CrModLA} implies
$$k = k \cdot 1_{\g} = k \cdot \pi(1_{\k}) = [k,1_{\k}]_{\k} = 0$$
for all $k \in \k$, so this definition is empty. However, the unitality is essential for a braided interpretation of crossed modules: the braiding~$\sigma_{Lei}$ encoding the Leibniz relation does involve the unit. The following classical construction provides a switch between non-unital and unital settings:

\begin{lemma}[Unitarization]\label{l:Unit}
$ $

Take a Leibniz algebra $(\k, \,[,]_{\k}\, )$ in $\Vect$.

\begin{enumerate}
\item A unital Leibniz algebra structure on~$\kp = \k \oplus \kk1$ is defined via
\begin{align*}
[k,\wk]_{\kp} &= [k,\wk]_{\k}, & [k,1]_{\kp} = [1,\wk]_{\kp} &= 0, &k,\wk \in \k.
\end{align*}
\item A cocommutative coassociative counital algebra structure on~$\kp$ is defined by putting
\begin{align*}
\Delta(k) &= k \otimes 1 + 1 \otimes k, & \varepsilon(k) &= 0, & k\in \k,\\
\Delta(1) &= 1 \otimes 1, &\varepsilon(1) &=1.&
\end{align*}
\item A Leibniz algebra morphism $f \colon \k \to \g$ extends to a unital Leibniz algebra morphism $f \colon \kp \to \gp$ by putting $f(1) = 1$.
\item A derivation~$f$ of a Leibniz algebra~$\k$ extends to a derivation of~$\kp$ via $f(1) = 0$. 
\item A $\k$-action~$\cdot$ on a vector space~$M$ extends to a $\kp$-action via $m \cdot 1 = m$.
\end{enumerate}

Take now a crossed module of Leibniz algebras $(\k,\g,\pi,\cdot)$. Consider the adjoint action $k \cdot \wk = [k,\wk]_{\k}$ of~$\k$ on itself. Extend it first into an action of~$\k$ on~$\kp$ by derivations, and then into an action of~$\kp$ on~$\kp$ as explained above. Explicitly, put $k \cdot 1 = k$, $1 \cdot k = \varepsilon(k)1$, $k \in \kp$. Similarly, unitarize the adjoint action of~$\g$ on itself and the $\g$-action on~$\k$ from the crossed module structure. Denote by~$\cdot$ all these unitarized actions. Further, extend the connecting map~$\pi$ into $\pi \colon \kp \to \gp$. Then one has
\begin{align}
k \cdot \pi(\wk) &= k \cdot \wk, & &k,\wk \in \kp, \label{eqn:CrModLAUn}\\
\pi (k \cdot g) &= \pi(k) \cdot g, & &k \in \kp,\, g \in \gp, \label{eqn:CrModLAUn'}\\
\pi (k \cdot \wk) &= \pi(k) \cdot \pi(\wk), & &k,\wk \in \kp, \label{eqn:CrModLAUn'''}\\
\Delta \circ \pi &= (\pi \otimes \pi) \circ \Delta \colon \kp \to \gp \otimes \gp.\label{eqn:CrModLAUn''}
\end{align}
\end{lemma}

The proof is straightforward. The comultiplication~$\Delta$ previously appeared in~\cite{CatSelfDistr,Lebed1}. Note that if a non-abelian Leibniz algebra~$\g$ carries a $\k$-action by derivations, the extended $\kp$-action from the lemma is no longer by derivations: the action by~$1$ behaves in the wrong way.

\begin{notation}
We use the same notation for a Leibniz algebra morphism / a derivation / an action and their unitarized versions from the lemma.
\end{notation}

\begin{proposition}[Crossed modules of Leibniz algebras as braided systems]\label{prop:CrModLABrSyst}
$ $ 
For a crossed module of Leibniz algebras $(\k,\g,\pi,\cdot)$, the following data define a rank $2$ braided system in~$\Vect$:
\begin{itemize}
\item as components, take $C=\kp$ and $A=\gp$;
\item the braiding is defined by
\begin{align*}
\sigma_{C,C} = \sigma_{coAss} \colon & 1 \otimes 1 \mapsto 1 \otimes 1, \quad 1 \otimes k \mapsto 1 \otimes k + k \otimes 1, \\
& k \otimes \wk \mapsto 0, \quad k \in \k, \wk \in \kp,\\
\sigma_{A,A} = \sigma_{Lei} \colon & g \otimes \wg \mapsto \wg \otimes g + 1 \otimes [g, \wg], \quad g,\wg \in \gp,\\
\sigma_{C,A} \colon & k \otimes g \mapsto g \otimes k \quad \text{if } k=1 \text{ or } g=1,\\ 
& k \otimes g \mapsto g \otimes k + 1 \otimes k \cdot g, \quad k \in \k, g \in \g.
\end{align*} 
\end{itemize}
\end{proposition}

\begin{remark}\label{rmk:AdjActionLA}
The unitarization procedure from Lemma~\ref{l:Unit} allows one to write the maps~$\sigma_{A,A}$ and~$\sigma_{C,A}$ in a uniform way: 
\begin{align*}
\sigma_{A,A} (\wg \otimes g) &= g_{(1)} \otimes \wg \cdot g_{(2)},  &g, \wg \in \gp,\\
\sigma_{C,A} (k \otimes g) &= g_{(1)} \otimes k \cdot g_{(2)}, &k \in \kp, g \in \gp,
\end{align*}
using Sweedler's notation $\Delta(g) = g_{(1)} \otimes g_{(2)}$.
\end{remark}

\begin{proof}
The cYBE on $C\otimes C \otimes C$ and $A\otimes A \otimes A$ are taken care of by Examples~\ref{ex:coAss} and~\ref{ex:Lei}. Both sides of the cYBE on $C\otimes C \otimes A$ equal
\begin{itemize}
\item $g \otimes 1 \otimes 1$ on $1 \otimes 1 \otimes g$, $g \in \gp$;
\item $1 \otimes (1 \otimes k + k \otimes 1)$ on $1 \otimes k \otimes 1$, $k \in \k$;
\item $g \otimes (1 \otimes k + k \otimes 1) + 1 \otimes (1 \otimes k \cdot g + k \cdot g \otimes 1)$ on $1 \otimes k \otimes g$, $k \in \k, g \in \g$;
\item $0$ on $k \otimes \wk \otimes g$, $k \in \k, \wk \in \kp, g \in \gp$.
\end{itemize}
The cYBE on $C\otimes A \otimes A$ is equivalent to
$$(k \cdot g) \cdot \wg = (k \cdot \wg) \cdot g + k \cdot [g,\wg], \qquad k \in \kp, g,\wg \in \gp,$$
which follows from the fact that the unitarization of the $\g$-action~$\cdot$ on~$\k$ is a $\gp$-action on~$\kp$. 
\end{proof}

\begin{lemma}\label{l:YDforCrModLA}
Take a YD module $(M,\ast,\delta)$ over the braided system above. Recall Sweedler's notations $\delta(m) = m_{(0)} \otimes m_{(1)}$, $(\delta \otimes \kp) \circ \delta(m) = m_{(0)} \otimes m_{(1)} \otimes m_{(2)}$. Consider also the map $f(m) = \varepsilon(m_{(1)})m_{(0)}$. Then one has the following relations:
\begin{align}
&(m \ast \wg) \ast g = (m \ast g_{(1)}) \ast (\wg \cdot g_{(2)}),  &&m \in M, \, g, \wg \in \gp,\label{eqn:YDforCrModLA}\\
&\delta(m \ast g) = m_{(0)} \ast g_{(1)} \otimes m_{(1)} \cdot g_{(2)},  &&m \in M, \, g \in \gp,\label{eqn:YDforCrModLA'}\\
& m_{(0)} \otimes m_{(1)} \otimes m_{(2)} = f(m_{(0)}) \otimes \Delta(m_{(1)}),  &&m \in M, \label{eqn:YDforCrModLA'''}\\
&(m \ast 1) \ast g = (m \ast g) \ast 1,  &&m \in M, \, g \in \gp,\label{eqn:YDforCrModLA''}\\
& f(m)_{(0)} \otimes f(m)_{(1)} = f(m_{(0)}) \otimes m_{(1)},  &&m \in M. \label{eqn:YDforCrModLA4}
\end{align}
\end{lemma}

\begin{proof}
The first three equations are the defining relations of generalized YD modules, with the braiding components written as suggested in Remark~\ref{rmk:AdjActionLA}. The penultimate relation follows from the first one by taking $\wg = 1$, and the last one from~\eqref{eqn:YDforCrModLA'''} by applying~$\varepsilon$ to the last component.
\end{proof}

We now propose a notion of representation of a crossed module of Leibniz algebras. It is tailored for admitting an interpretation in terms of generalized YD modules.

\begin{definition}[Representations of a crossed module of Leibniz algebras]\label{def:CrModLARep}
$ $

A \emph{representation of a crossed module of Leibniz algebras} $(\k,\g,\pi,\cdot)$ is a vector space~$M$ endowed with a $\g$-action~$\ast$ and a linear map~$\delta_0 \colon M \to M \otimes \k$ which is
\begin{itemize}
\item of square zero, i.e. $(\delta_0 \otimes \k) \circ \delta_0 = 0$, and
\item $\g$-equivariant, in the sense of $\delta_0(m \ast g) = \delta_0(m) \ast g$, where $\g$ acts on $M \otimes \k$ according to the Leibniz rule: $(m \otimes k) \ast g = m \otimes k \cdot g + m \ast g \otimes k$.
\end{itemize}
The category of such representations (with the obvious notion of morphisms) is denoted by $\MMM(\k,\g,\pi,\cdot)$, or simply $\MMM(\k,\g)$.
\end{definition}

\begin{remark}
If $\k$ has a basis $k_i, i  \in I$, then the map~$\delta_0$ can be written as $\delta_0(m) = \sum_{i \in I} \theta_i(m) \otimes k_i$ for some linear maps $\theta_i \colon M \to M$. The square-zero property for~$\delta_0$ then reads $\theta_i \theta_j = 0$ for all $i,j \in I$. Moreover, in the finite-dimensional case, the $\g$-equivariance yields an expression of $\theta_i(m \ast g) - \theta_i(m) \ast g$ in terms of the $\theta$s and the structural constants of the action of~$\g$ on~$\k$.
\end{remark}

\begin{proposition}[Representations of a crossed module of Leibniz algebras as generalized YD modules]\label{prop:CrModLACatIso}
$ $

In the settings of the previous proposition, one has category inclusions
\begin{align*}
\MMM(\k,\g) & \hookrightarrow \YD_{\gp}^{\kp},\\
(M,\ast,\delta_0) &\mapsto (M,\ast,\delta),
\end{align*}
where the $\g$-action~$\ast$ on~$M$ is extended to a $\gp$-action as explained in Lemma~\ref{l:Unit}, and the $\kp$-coaction~$\delta$ is given by $\delta(m) = \delta_0(m) + m \otimes 1$.
\end{proposition}

\begin{proof}
As recalled in Example~\ref{ex:Lei}, the $\gp$-action~$\ast$ on~$M$ is also a $(\gp;\sigma_{Lei})$-action. Further, one verifies that a linear map $\delta \colon M \to M \otimes \kp$ defines a $(\kp;\sigma_{coAss})$-coaction~$\delta$, normalized in the sense of $(M \otimes \varepsilon) \circ \delta = M$, if and only if it has the form $\delta(m) = \delta_0(m) + m \otimes 1$, with $\delta_0 \colon M \to M \otimes \k$ of square zero. At last, the YD property~\eqref{eqn:gen_YD} for our~$\sigma_{C,A}$ is equivalent to the $\g$-equivariance of~$\delta_0$. Thus the functor from the theorem is well defined on objects. One easily sees that it is well defined, full and faithful on morphisms. Finally, the $\g$-action on~$M$ can be restored from the $\gp$-action, and the map~$\delta_0$ from~$\delta$, hence our functor is indeed a category inclusion.
\end{proof}

As usual, one can interpret the whole category $\YD_{\gp}^{\kp}$ in terms of non-normalized representations; the details are left to the Reader.

Now, Theorem~\ref{thm:genYD_br} allows one to construct braidings:

\begin{theorem}[Representations of a crossed module of Leibniz algebras are braided]\label{thm:SrModLABr} $ $

Any representations $(M_i, \ast_i, (\delta_0)_i)$ 
of a crossed module of Leibniz algebras $(\k,\g,\pi,\cdot)$ form a braided system, with the braidings on $M_i \otimes M_j$ given by
\begin{align}\label{eqn:SigmaCrModLA}
\sigma_{CrModLA}(m \otimes n)& = n \otimes m + n_{(0)} \otimes m \ast_i \pi(n_{(1)}),
\end{align} 
using Sweedler's notation $(\delta_0)_j(n) = n_{(0)} \otimes n_{(1)}$.
\end{theorem}

\begin{proof}
We will check the technical condition~\eqref{eqn:cond_pi} for our map~$\pi$, with $\alpha_2 = \gamma_1 = \gamma_2 = 1$ and $\alpha_1 = 0$. It reads
\begin{align*}
&\pi(\wk_{(2)}) \otimes \pi(\wk_{(1)}) \otimes \pi(k \cdot \pi(\wk_{(3)})) =\\ 
&\pi(\wk_{(1)}) \otimes (\pi(\wk_{(2)}))_{(1)} \otimes \pi(k) \cdot (\pi(\wk_{(2)}))_{(2)}
\end{align*}
for $k, \wk \in \kp$, using the usual Sweedler's notation for the comultiplications on~$\kp$ and on~$\gp$. Since these comultiplications are cocommutative and are entwined by~$\pi$ (relation~\eqref{eqn:CrModLAUn''}), it suffices to show that
\begin{align*}
\pi(k \cdot \pi(\wk)) &= \pi(k) \cdot \pi(\wk),
\end{align*}
which follows from~\eqref{eqn:CrModLAUn'}. 
\end{proof}

\begin{example}[Adjoint representations]\label{ex:AdjRepLA}
$ $

Recall that, for a Leibniz algebra~$\k$, the identity map $\Id_{\k} \colon \k \to \k$ and the adjoint action $k \cdot \wk = [k,\wk]_{\k}$ define a crossed module structure. Moreover, $\k$ itself with the map~$\delta_0$ and again the adjoint action~$\ast$ is a representation of this crossed module. Theorem~\ref{thm:SrModLABr} then endows~$\k$ with a braiding, which turns out to be the flip $k \otimes \wk \mapsto \wk \otimes k$. Further, $\kp$ with $\delta_0$ defined by $\delta_0(1) = 0$ and $\delta_0(k) = 1 \otimes k$ for $k \in \k$ is also a representation of this crossed module. The braiding recovered in this latter case is the Leibniz braiding~$\sigma_{Lei}$.
\end{example}

\section{Categorical aspects}\label{sec:AdjMod}

This section is devoted to a systematic construction of families of generalized YD modules, and to a study of the categories~$\YD_A^C$. We return here to the general setting of a strict monoidal category~$\C$.

First we describe a method for transforming generalized YD modules into more complicated ones.

\begin{proposition}[Enrichment of YD modules]\label{prop:AdjMod}
$ $

Take a YD module $(N,\rho,\delta)$ over a braided system $(C,A;\osigma)$ in~$\C$. Suppose that this system can be enriched into a rank~$3$ system $(C,M,A;$ $\osigma,\sigma_{C,M},$ 
$\sigma_{M,M},\sigma_{M,A})$. Then $N \otimes M$ can be endowed with the following YD module structure over $(C,A;\osigma)$ (Fig.~\ref{pic:enrichYD}):
\begin{align*}
\delta' &= (N \otimes \sigma_{C,M}) \circ (\delta \otimes M),&
\rho' &= (\rho \otimes M) \circ (N \otimes \sigma_{M,A}).
\end{align*}
\end{proposition}

\begin{center}
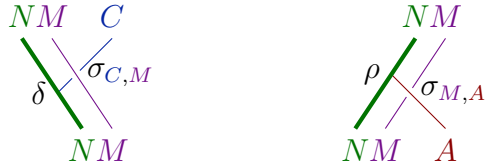

\begin{tikzpicture}[scale=0.4]
\draw [myblue] (1.19,1.19)--(1.67,1.67);
\draw [myblue] (1.93,1.93)--(3,3);
\draw [myolive, ultra thick] (2,0)--(0,3);
\draw [myviolet] (3,0)--(1,3);
\node at (2,0) [below,myolive] {$N$}; 
\node at (3,0) [below,myviolet] {$M$}; 
\node at (0,3) [above,myolive] {$N$}; 
\node at (1,3) [above,myviolet] {$M$}; 
\node at (3,3) [above,myblue] {$C$}; 
\node at (1.19,1.19) [left] {$\delta$}; 
\node at (1.8,1.8) [right] {$\sigma_{{\color{myblue} C},{\color{myviolet} M}}$}; 
\end{tikzpicture}
\hspace*{2cm}
\begin{tikzpicture}[scale=0.4]
\draw [myred] (3,0)--(1.19,1.81);
\draw [myolive, ultra thick] (0,0)--(2,3);
\draw [myviolet] (1,0)--(1.7,1.05);
\draw [myviolet] (1.9,1.35)--(3,3);
\node at (0,0) [below,myolive] {$N$}; 
\node at (1,0) [below,myviolet] {$M$}; 
\node at (2,3) [above,myolive] {$N$}; 
\node at (3,3) [above,myviolet] {$M$}; 
\node at (3,0) [below,myred] {$A$};
\node at (1.19,1.81) [left] {$\rho$}; 
\node at (1.8,1.2) [right] {$\sigma_{{\color{myviolet} M},{\color{myred} A}}$}; 
\end{tikzpicture}
\captionof{figure}{Enriched YD modules}\label{pic:enrichYD}
\end{center}

\begin{proof}
We have to show the braided module and comodule property for $N\otimes M$ and 
the Yetter-Drinfel$'$d property.
   
\noindent{\bf 1)} The claim of the braided module property is the equality of
$$(\rho\otimes M)\circ(N\otimes\sigma_{M,A})\circ(\rho\otimes M\otimes A)\circ
(N\otimes\sigma_{M,A}\otimes A)\circ(N\otimes M\otimes\sigma_{A,A})$$
and 
$$(\rho\otimes M)\circ(N\otimes\sigma_{M,A})\circ(\rho\otimes M\otimes A)
\circ(N\otimes\sigma_{M,A}\otimes A).$$
For this, one uses first Equation~\eqref{eqn:cYBE} for $M \otimes A\otimes A$, and then the defining property~\eqref{eqn:BrMod} for the braided $A$-module $(N,\rho)$.

\noindent{\bf 2)} One argues similarly for the braided comodule property (using~\eqref{eqn:cYBE} for 
$C\otimes C \otimes M$).
 
\noindent{\bf 3)}
The YD property~\eqref{eqn:gen_YD} reads
$$(N\otimes\sigma_{C,M})\circ(\delta\otimes M)\circ(\rho\otimes M)\circ
(N\otimes\sigma_{M,A})=$$
$$(\rho\otimes M\otimes C)\circ(N\otimes\sigma_{M,A}\otimes C)\circ
(N\otimes M\otimes\sigma_{C,A})\circ(N\otimes\sigma_{C,M}\otimes A)\circ
(\delta\otimes M\otimes A).$$
It follows from the cYBE for $C\otimes M\otimes A$, and then Equation~\eqref{eqn:gen_YD} for the YD module
$(N,\rho,\delta)$.

The reader is invited to draw the corresponding diagrams.    
\end{proof}

Note that the datum of~$\sigma_{M,M}$ is completely irrelevant for the YD module structure on $N \otimes M$, and can be replaced, for instance, with~$\Id_{M \otimes M}$ (which trivially forces all the instances of the cYBE involving at least two copies of~$M$). This motivates the following

\begin{definition}[Enriching structures]\label{def:center}
$ $

Take a braided system $(C,A;\osigma)$ in~$\C$. Denote by~$\ZZZ_A^C$ the category whose
\begin{itemize}
\item objects are the \emph{enriching structures} for $(C,A;\osigma)$, i.e., objects~$M$ together with morphisms~$\sigma_{C,M}$ and~$\sigma_{M,A}$ in~$\C$ such that $(C,M,A;$ $\osigma,\sigma_{C,M},\Id_{M \otimes M},\sigma_{M,A})$ is a braided system;
\item morphisms are those morphisms~$\varphi \colon M \to M'$ in~$\C$ which satisfy the naturality conditions
\begin{align}
(\varphi \otimes C) \circ \sigma_{C,M} &= \sigma_{C,M'} \circ (C \otimes \varphi),\label{eqn::Nat1}\\
(A \otimes \varphi) \circ \sigma_{M,A} &= \sigma_{M',A} \circ (\varphi \otimes A).\label{eqn::Nat2}
\end{align}
\end{itemize}
\end{definition}

This notion is related to the categorical center (hence the notation), and to factorisations of a distributive law, introduced by U.~Kr\"{a}hmer and P.~Slevin \cite{KraSle} as a tool for constructing new cyclic homology theories.

We now show that the category~$\ZZZ_A^C$ is far from being empty:

\begin{proposition}[Categorical aspects of the enrichment]\label{prop:center}
$ $

\begin{enumerate}
\item The category~$\ZZZ_A^C$ is strict monoidal: the tensor product structure is given by the tensor product of~$\C$, together with
\begin{align*}
\sigma_{C,M \otimes M'} &= (M \otimes \sigma_{C,M'}) \circ (\sigma_{C,M} \otimes {M'}),\\
\sigma_{M \otimes M',A} &= (\sigma_{M,A} \otimes {M'}) \circ (M \otimes \sigma_{M',A}),
\end{align*}
and the unit object is~$\II$ with $\sigma_{C,\II} = \Id_C$ and $\sigma_{\II,A} = \Id_A$ (Fig.~\ref{pic:enrich_tensor}).
\item The category~$\ZZZ_A^C$ contains $(A,\sigma_{C,A},\sigma_{A,A})$ and $(C,\sigma_{C,C},\sigma_{C,A})$, as well as all their mixed tensor products.
\item Proposition~\ref{prop:AdjMod} yields a bifunctor 
\begin{align*}
E_A^C \colon \YD_A^C \times \ZZZ_A^C &\to \YD_A^C;
\end{align*}
on morphisms, it is defined by $\varphi \times \psi \mapsto \varphi \otimes \psi$.
\end{enumerate}
\end{proposition}

\begin{center}
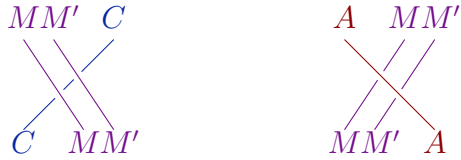

\begin{tikzpicture}[scale=0.4]
\draw [myblue] (0,0)--(1.05,1.05);
\draw [myblue] (1.33,1.33)--(1.67,1.67);
\draw [myblue] (1.93,1.93)--(3,3);
\draw [myviolet] (2,0)--(0,3);
\draw [myviolet] (3,0)--(1,3);
\node [myblue]  at (0,-0.4){$C$};
\node [myviolet]  at (2,-0.4){$M$};
\node [myviolet]  at (3.2,-0.4){$M'$};
\node at (0,3) [above,myviolet] {$M$}; 
\node at (1.2,3) [above,myviolet] {$M'$}; 
\node at (3,3) [above,myblue] {$C$}; 
\end{tikzpicture}
\hspace*{2cm}
\begin{tikzpicture}[scale=0.4]
\draw [myred] (3,0)--(0,3);
\draw [myviolet] (0,0)--(1.1,1.65);
\draw [myviolet] (1.3,1.95)--(2,3);
\draw [myviolet] (1,0)--(1.7,1.05);
\draw [myviolet] (1.9,1.35)--(3,3);
\node [myred]  at (3,-0.4){$A$};
\node [myviolet]  at (0,-0.4){$M$};
\node [myviolet]  at (1.2,-0.4){$M'$};
\node at (2,3) [above,myviolet] {$M$}; 
\node at (3.2,3) [above,myviolet] {$M'$}; 
\node at (0,3) [above,myred] {$A$};
\end{tikzpicture}
\captionof{figure}{Tensor product of two enriching structures}\label{pic:enrich_tensor}
\end{center}

\begin{proof}
\noindent{(1)}
Equation~\eqref{eqn:cYBE} on $C \otimes (M \otimes M') \otimes A$ reads
$$(\sigma_{M,A}\otimes M'\otimes A)\circ(M\otimes\sigma_{M',A}\otimes C) \circ 
(M\otimes M'\otimes\sigma_{C,A}) \circ {\color{white} A}$$
$$(M\otimes\sigma_{C,M'}\otimes A)\circ (\sigma_{C,M}\otimes M'\otimes A)=$$
$$(A\otimes M\otimes\sigma_{C,M'})\circ(A\otimes\sigma_{C,M}\otimes M') \circ
(\sigma_{C,A}\otimes M\otimes M')\circ {\color{white} A}$$
$$(C\otimes\sigma_{M,A}\otimes M')\circ (C\otimes M\otimes\sigma_{M',A}).$$
In these two expressions, one recognizes in the interior by leaving apart both exterior 
$\sigma$-expressions an expression involving only $\sigma_{M',A}$, $\sigma_{C,A}$, and
$\sigma_{C,M'}$ (respectively, only $\sigma_{C,M}$, $\sigma_{C,A}$, and $\sigma_{M,A}$). On these 
expressions, one may apply~\eqref{eqn:cYBE} for $C \otimes M' \otimes A$ (or $C \otimes M \otimes A$). 
The resulting total expressions are identical. Equation~\eqref{eqn:cYBE} on $C \otimes C \otimes (M \otimes M')$ and on $(M \otimes M') \otimes A \otimes A$ is treated similarly. As usual, drawing pictures can be helpful for following the arguments above. One concludes that $M\otimes M'$ is an enriching structure.  
Strict associativity and strict unitality are clear. 

 Claims (2) and (3) are immediate. 
\end{proof}

\begin{remark}\label{rmk:ActCoactCompat}
When saying that condition~\eqref{eqn:gen_YD} means that~$\delta$ is a morphism in $\ModCat_{(A;\sigma_{A,A})}$ or, equivalently, $\rho$ is a morphism in $\ModCat^{(C;\sigma_{C,C})}$, we endowed $M \otimes C$ and $M \otimes A$ with the structures from Propositions~\ref{prop:AdjMod} and~\ref{prop:center}.
\end{remark}

We next present a toy type of YD modules often encountered in practice; enriched according to Proposition~\ref{prop:AdjMod}, they yield an important source of meaningful examples of generalized YD modules.

\begin{definition}\label{def:char_YD} The unit object~$\II$ of~$\C$ endowed with a YD module structure over a rank~$2$ braided system in~$\C$ is called a \emph{Yetter-Drinfel$'$d character} of the system.
\end{definition}

Concretely, a YD character structure over $(C,A;\osigma)$ includes a braided character $\varepsilon_A \colon A \to \II$ over $(A; \sigma_{A,A})$ and a braided cocharacter  $\nu_C \colon \II \to C$  over $(C; \sigma_{C,C})$, compatible in the sense of
\begin{align}\label{eqn:YDChar}
\nu_C \circ \varepsilon_A &= (\varepsilon_A \otimes C) \circ  \sigma_{C,A} \circ (\nu_C \otimes A).
\end{align}

According to Proposition~\ref{prop:center}, a YD character permits to endow~$A$, $C$, and all their mixed tensor products with a YD module structure over $(C,A;\osigma)$, which we call \emph{adjoint} because of the following examples.

\begin{example}[Woronowicz and Hennings braidings for a Hopf algebra]\label{ex:Hopf}
$ $

Consider the braided system from Example~\ref{ex:YD'}. In this case, a usual character of the algebra $(H,\mu,\nu)$ (i.e., a morphism $\zeta \colon H \to \II$ satisfying $\zeta \circ \mu = \zeta \otimes \zeta$ and $\zeta \circ \nu = \Id_{\II}$) is automatically a braided character over $(A; \sigma_{A,A})$. Similarly, a usual cocharacter $\eta \colon \II \to H$ of $(H,\varepsilon,\Delta)$ is a braided cocharacter over $(C; \sigma_{C,C})$. The compatibility condition~\eqref{eqn:YDChar} becomes here
\begin{align}\label{eqn:YDchar_Hopf}
&\mu^2 \circ (S \otimes (\eta \circ \zeta) \otimes H) \circ \Delta^2 = \eta \circ \zeta.
\end{align}

In the case $\zeta = \varepsilon$ and $\eta = \nu$, it follows from the definition of the antipode. Feeding the YD character $(\varepsilon, \nu)$ and the object~$H$ viewed either as~$C$ or as~$A$ into Proposition~\ref{prop:AdjMod}, one obtains two generalized YD module structures on~$H$. We have seen that $\pi = \Id_H$ satisfies condition~\eqref{eqn:cond_pi}. Theorem~\ref{thm:genYD_br} thus yields two different braidings on~$H$:
\begin{align*}
\sigma_{H} &= (H \otimes \mu^2) \circ ({H} \otimes c_{H,H} \otimes {H}) \circ (c_{H,H} \otimes S \otimes {H}) \circ (H \otimes \Delta^2),\\
\sigma'_{H} &= (H \otimes \mu^2)  \circ (c_{H,H} \otimes S \otimes {H}) \circ ({H} \otimes c_{H,H} \otimes {H}) \circ (H \otimes \Delta^2).
\end{align*}
These braidings are in categorical duality (note that the axioms defining a Hopf algebra, as well as the cYBE, are self-dual). Pictorially, this duality is reflected in the horizontal symmetry of the corresponding diagrams. 
In~$\Vect$, these braidings read
\begin{align*}
\sigma_{H}(h \otimes h') & = h'_{(1)} \otimes S(h'_{(2)})hh'_{(3)},\\
\sigma'_{H}(h \otimes h') & = h'_{(2)} \otimes hS(h'_{(1)})h'_{(3)},
\end{align*}
which are precisely the formulas discovered by S.L.~Woronowicz in~\cite{Wor}. We thus include the results of~\cite{Wor}, which seemed mysterious at the time, into a general conceptual framework.

For a general character-cocharacter pair $(\zeta, \eta)$, condition~\eqref{eqn:YDchar_Hopf} may fail. However, it becomes true when pre-composed with~$\nu$ or post-composed with~$\varepsilon$. Hence condition~\eqref{eqn:YDChar} (tensored with $\Id_H$ on the right) holds true when pre-composed  with~$\sigma_{A,A} = \sigma_{Ass}$ or post-composed with~$\sigma_{C,C} = \sigma_{coAss}$. But this is sufficient for Proposition~\ref{prop:AdjMod} to produce generalized YD module structures on~$H$---and hence for Theorem~\ref{thm:genYD_br} to produce braidings on~$H$. These braidings are obtained from~$\sigma_{H}$ and~$\sigma'_{H}$ by replacing~$\Delta^2$ with $(H \otimes H \otimes ((\zeta \otimes H) \circ \Delta)) \circ \Delta^2$---or, respectively, by replacing~$\mu^2$ with $\mu^2 \circ (H \otimes H \otimes (\mu \circ (\eta \otimes H)))$. 
In~$\Vect$, we recover the braidings of M.A.~Hennings~\cite{Hennings}:
\begin{align*}
\sigma_{H}(h \otimes h') & = \zeta(h'_{(3)}) h'_{(1)} \otimes S(h'_{(2)})hh'_{(4)},\\
\sigma'_{H}(h \otimes h') & = h'_{(2)} \otimes hS(h'_{(1)})\eta h'_{(3)}.
\end{align*}
Note that all modules and comodules appearing in these constructions are normalized, and so one actually gets usual YD module structures on~$H$.
\end{example}

We now show that in our favourite examples, all generalized YD modules can be found inside the category~$\ZZZ_A^C$.

\begin{proposition}[YD modules as enriching structures]\label{YDModAsEnrich}
$ $

Take a Hopf algebra $(H,\mu,\nu,\varepsilon,\Delta,S)$ in a symmetric monoidal category $(\C,\otimes,\II,c)$, and consider the braided system $(H,H; \osigma)$ from Example~\ref{ex:YD'}. Then the (non-normalized) YD modules over~$H$ can be seen as a full subcategory of~$\ZZZ_H^H$ via the functor
\begin{align*}
Z_{YD} \colon \YD_H^H &\hookrightarrow \ZZZ_H^H,\\
(M,\rho,\delta) &\mapsto (M,\sigma_{H,M},\sigma_{M,H}),\\
\text{where} \qquad &\sigma_{H,M} = (M \otimes \mu) \circ (c_{H,M} \otimes H) \circ (H \otimes \delta),\\
&\sigma_{M,H} = (H \otimes \rho) \circ (c_{M,H} \otimes H) \circ (M \otimes \Delta).
\end{align*}
\end{proposition}

The braided systems thus obtained are related to, but different from, those studied in \cite{Lebed,Lebed2ter}.

\begin{proof}
In order to show that one indeed gets enriching structures, one has to check $3$ instances of the cYBE. The graphical calculus works well here; we leave the tedious but straightforward verifications to the Reader. Further, note that the YD structure on~$M$ can be reconstructed from the $\sigma$'s via
\begin{align}\label{eqn:ZtoYD}
\rho &= (\varepsilon \otimes M) \circ \sigma_{M,H}, & 
\delta &= \sigma_{H,M} \circ (\nu \otimes M). 
\end{align}
This proves that~$Z_{YD}$ is injective on objects. These formulas also show that the naturality condition~\eqref{eqn::Nat1} for a morphism $\varphi \colon M \to M'$ in~$\C$ is equivalent to $\varphi$ being a morphism of comodules, while~\eqref{eqn::Nat2} is equivalent to $\varphi$ being a morphism of modules. Thus the functor~$Z_{YD}$ is well-defined and fully faithful on morphisms. 
\end{proof}

\begin{remark}
In fact, formulas~\eqref{eqn:ZtoYD} define a functor $M_{YD} \colon \ZZZ_H^H \to \YD_H^H$ such that $M_{YD} \circ Z_{YD}$ is the identity functor on~$\YD_H^H$. It would be interesting to understand how far these functors are from category equivalences.
\end{remark}

Combining~$Z_{YD}$ with the functor $E_H^H \colon \YD_H^H \times \ZZZ_H^H \to \YD_H^H$ from Proposition~\ref{prop:center}, one obtains a bifunctor
$\YD_H^H \times \YD_H^H \to \YD_H^H$. Together with the YD character $(I,\varepsilon,\nu)$ (Example~\ref{ex:Hopf}), they define a tensor structure on~$\YD_H^H$, extending the classical tensor structure on~$\YDn_H^H$ to the non-normalized setting. (The vocabulary of tensor categories is recalled in Definition~\ref{def:tensor}.) Concretely, the tensor product of two YD modules $(M,\rho,\delta)$ and $(M',\rho',\delta')$ in $\YD_H^H$ is the object $M \otimes M'$ with
\begin{align}
\rho_{M \otimes M'} &= (\rho \otimes \rho') \circ (M \otimes c_{M',H} \otimes H) \circ ({M \otimes M'} \otimes \Delta),\label{eqn:tens_YD}\\
\delta_{M \otimes M'} &= ({M \otimes M'} \otimes \mu)  \circ (M \otimes c_{H,M'} \otimes H) \circ (\delta \otimes \delta').\label{eqn:tens_YD'}
\end{align}
The maps~$\sigma_{gYD}$ become morphisms in this category. Even better: they provide a braided structure on~$\YD_H^H$.

\begin{proposition}[Representations of a crossed module of groups as enriching structures]\label{RepCrModGrAsEnrich}
$ $

Take a crossed module of groups $(K, G, \pi, \cdot)$, and consider the braided system $(\kk K, \kk G; \osigma)$ from Example~\ref{ex:CrMod'}. Then the YD modules over this system (and, in particular, Bantay's  representations of our crossed module) can be seen as a full subcategory of~$\ZZZ_{\kk G}^{\kk K}$ via the functor
\begin{align*}
Z_{CrMod} \colon \YD_{\kk G}^{\kk K} &\hookrightarrow \ZZZ_{\kk G}^{\kk K},\\
(M,\rho,\delta) &\mapsto (M,\sigma_{\kk K,M},\sigma_{M,\kk G}),\\
\text{where} \qquad &\sigma_{\kk K,M} (k,m) = \sum\nolimits_{k' \in K} m_{k'} \otimes kk',\\
&\sigma_{M,\kk G}(m,g) = (g, m\ast g),
\end{align*}
the coaction~$\delta$ is written as $\delta(m)= \sum_{k \in K} m_k \otimes k$, and~$\ast$ denotes the action~$\rho$.
\end{proposition}

The proof is similar to that of Proposition~\ref{YDModAsEnrich}. The functor~$Z_{CrMod}$ admits a left inverse~$M_{CrMod}$, defined by formulas analogous to~\eqref{eqn:ZtoYD}. Combining~$Z_{CrMod}$ with the functor $E_{\kk G}^{\kk K}$, one obtains a bifunctor $\YD_{\kk G}^{\kk K} \times \YD_{\kk G}^{\kk K} \to \YD_{\kk G}^{\kk K}$, yielding a monoidal structure on~$\YD_{\kk G}^{\kk K}$. Explicitly, the tensor product of two YD modules over $(\kk K, \kk G; \osigma)$ is endowed with diagonal action and coaction, in the spirit of \eqref{eqn:tens_YD}-\eqref{eqn:tens_YD'}. The braidings~$\sigma_{gYD}$ enrich this monoidal category into a braided one.

\begin{proposition}[Twisted representations of a crossed module of shelves as enriching structures]\label{TwRepsAsEnrich}
$ $

Take a crossed module of shelves $(R,S,\pi,\cdot)$, and consider the braided system $(R,S; \osigma)$ from Proposition~\ref{prop:CrModShBrSyst}.

\begin{enumerate}
\item  The twisted representations of our crossed module can be seen as a full subcategory of~$\ZZZ_S^R$ via the functor
\begin{align*}
Z_{SD} \colon \MMM^{\operatorname{tw}}_{\Set}(R,S) \simeq \YD_S^R &\hookrightarrow \ZZZ_S^R,\\
(M, \mop, gr,f) &\mapsto (M,\sigma_{R,M},\sigma_{M,S}),\\
\text{where} \qquad &\sigma_{R,M} (r \otimes m) = f(m) \otimes gr(m) ,\\
&\sigma_{M,S} (m \otimes s) = s \otimes m \mop s.
\end{align*}
\item Alternatively, enriching structures can be constructed out of twisted representations via the functor
\begin{align*}
\widetilde{Z}_{SD} \colon \MMM^{\operatorname{tw}}_{\Set}(R,S) \simeq \YD_S^R &\to \ZZZ_S^R,\\
(M, \mop, gr,f) &\mapsto (M,\sigma_{R,M},\sigma_{M,S}),\\
\text{where} \qquad &\sigma_{R,M} (r \otimes m) = f(m) \otimes r \op gr(m) ,\\
&\sigma_{M,S} (m \otimes s) = s \otimes m \mop s.
\end{align*}
\end{enumerate}
Similar functors exist in the linear setting.
\end{proposition}

\begin{proof}
We treat only the set-theoretic case here; the linear case is similar.

\begin{enumerate}
\item As usual, one has to check $3$ instances of the cYBE. We do it here by explicit calculations.
\begin{itemize}
\item On $R \otimes R \otimes M$, the cYBE takes the form
$$f^2(m) \otimes gr(f(m)) \otimes gr(m) = f^2(m) \otimes gr(m) \otimes gr(m),$$
which is equivalent to~$f$ preserving the $R$-grading.
\item On $R \otimes M \otimes S$, the cYBE becomes
$$s \otimes f(m) \mop s \otimes gr(m)\cdot s = s \otimes f(m \mop s) \otimes gr(m \mop s),$$
which is equivalent to the $S$-actions intertwining both~$f$ and~$gr$.
\item On $M \otimes S \otimes S$, the cYBE reads
$$s' \otimes s \op s' \otimes (m \mop s) \mop s'  = s' \otimes s \op s' \otimes (m \mop s') \mop (s \op s'),$$
which is equivalent to~$\mop$ being an $S$-action.
\end{itemize}

Further, the maps~$f$ and~$gr$ can be reconstructed from~$\sigma_{R,M}$, and the $S$-action~$\mop$ from~$\sigma_{M,S}$. Moreover, the naturality condition~\eqref{eqn::Nat1} for a morphism $\varphi \colon M \to M'$ in~$\C$ is equivalent to $\varphi$ respecting the $R$-grading and intertwining~$f$ and~$f'$, while~\eqref{eqn::Nat2} is equivalent to $\varphi$ being a morphism of $S$-modules. Thus the functor~$Z_{SD}$ is well-defined and fully faithful on morphisms, and injective on objects.

\item The cYBE on $R \otimes R \otimes M$ and $R \otimes M \otimes S$ become here
$$f^2(m) \otimes r\op gr(f(m)) \otimes r\op gr(m) = f^2(m) \otimes r\op gr(m) \otimes r\op gr(m),$$
$$s \otimes f(m) \mop s \otimes (r\op gr(m))\cdot s = s \otimes f(m \mop s) \otimes (r \cdot s)\op gr(m \mop s).$$
They follow from the defining properties of a twisted representation. On $M \otimes S \otimes S$, the cYBE is the same as in the previous case. Further, the naturality condition~\eqref{eqn::Nat1} for a morphism $\varphi \colon M \to M'$ in~$\C$ follows from (but is not equivalent to!) $\varphi$ respecting the $R$-gradings and intertwining~$f$ and~$f'$, while~\eqref{eqn::Nat2} is equivalent to $\varphi$ being a morphism of $S$-modules. One thus has a well-defined functor. \qedhere
\end{enumerate} 
\end{proof}

The existence of two braided system structures on $(R,M,S)$ is a general phenomenon in the world of shelves; we already met it when observing two braided system structures on $(R,S)$ (Remark~\ref{rmk:CrModShBrSyst}).

\begin{remark}
The map~$\sigma_{R,M}$ used for constructing~$\widetilde{Z}_{SD}$ is in fact the braiding $\sigma_{TwCrModSh}$ for the twisted representations $(R, \cdot, \Id_R, \Id_R)$ and $(M, \mop,$ 
$gr, f)$ of the crossed module of shelves $(R,S,\pi,\cdot)$, since one has 
$$r \op gr(m) = r \cdot \pi(gr(m)).$$
\end{remark}

Combining~$Z_{SD}$ or~$\widetilde{Z}_{SD}$ with the functor~$E^R_S$ from Proposition~\ref{prop:center}, one obtains two bifunctors
\begin{align}\label{eqn:TensorRepCrModSh}
\otimes &= E^R_S \circ(\Id \times Z_{SD}),& \wotimes &= E^R_S \circ(\Id \times \widetilde{Z}_{SD}) 
\end{align}
from $\MMM^{\operatorname{tw}}_{\Set}(R,S) \times \MMM^{\operatorname{tw}}_{\Set}(R,S)$ to $\MMM^{\operatorname{tw}}_{\Set}(R,S)$.
The corresponding product structures are explicitly written as follows:

\begin{proposition}[Products of twisted representations of a crossed module of shelves]\label{ProdRepsCrModSh}
$\quad$

Given two twisted representations $(M, \mop, gr,f)$ and $(M', \mop', gr',f')$  in $\MMM^{\operatorname{tw}}_{\Set}(R,S)$, their product $M \otimes M'$ can be seen as a twisted representation in two different ways. In both cases the $S$-actions and the twisting maps are assembled diagonally:
\begin{align*}
(m \otimes m') \mop_{\otimes} s = (m \otimes m') \mop_{\widetilde{\otimes}} s &= m \mop s \otimes m' \mop' s, \\
f_{\otimes} (m \otimes m') = f_{\widetilde{\otimes}} (m \otimes m') &= f(m) \otimes f'(m').
\end{align*}
The $R$-gradings can be assembled either peripherally or diagonally:
\begin{align*}
gr_{\otimes}(m \otimes m') &= gr'(m'), &  gr_{\widetilde{\otimes}}(m \otimes m') &= gr(m) \op gr'(m').
\end{align*}
Similar structures exist in the linear setting.
\end{proposition}

It is natural to ask if any of these functors is a part of a monoidal structure on $\MMM^{\operatorname{tw}}_{\Set}(R,S)$. This question is more subtle here than in the case of usual YD modules or representations of a crossed module of groups. We now show that one gets something close to a monoidal category, and study the place of the braidings $\sigma_{TwCrModSh}$ (Remark~\ref{rmk:BrYDSR}) in this category. To give precise assertions, some definitions from category theory are first due.

\begin{definition}[Categorical vocabulary]\label{def:tensor}
$ $

\begin{itemize}
\item A \emph{pre-tensor category} is a category $\C$ endowed with a \emph{tensor product}, i.e., a bifunctor $\otimes \colon \C\times \C \to \C$ and natural isomorphisms 
$$\big(\,\alpha_{U,V,W} \colon (U \otimes V) \otimes W \overset{\sim}{\to} U \otimes (V \otimes W) \,\big)_{U,V,W \in \Ob(\C)},$$ 
called \emph{associator}, or \emph{associativity constraint}, satisfying the \emph{pentagon axiom}
\begin{equation}\label{eqn:Pentagon}
\xymatrix@R=0.25cm @C=0.6cm{
& ((U \otimes V) \otimes W) \otimes X \ar[dl]_-{\alpha_{U,V,W} \otimes X\quad} \ar[ddr]^{\alpha_{U \otimes V,W,X}} & \\
(U \otimes (V \otimes W)) \otimes X \ar[dd]_{\alpha_{U,V \otimes W,X} } & &  \\
& & (U \otimes V) \otimes (W \otimes X) \ar[ddl]^{\alpha_{U,V,W \otimes X}}\\
U \otimes ((V \otimes W) \otimes X) \ar[dr]_-{U \otimes \alpha_{V,W,X}\quad}& &  \\
& U \otimes (V \otimes (W \otimes X))& 
}
\end{equation}

\item A pre-tensor category is called \emph{tensor}, or \emph{monoidal}, if endowed with a \emph{unit}, i.e., an object~$I$ in~$\C$ and natural isomorphisms
$$\big(\,\lambda_{V} \colon I \otimes V \overset{\sim}{\to} V, \; \rho_{V} \colon V \otimes I \overset{\sim}{\to} V\,\big)_{V \in \Ob(\C)},$$ 
called \emph{left and right unitors}, or a \emph{unit constraint}, satisfying the \emph{triangle axiom}
\begin{equation}\label{eqn:Triangle}
\xymatrix@R=0.25cm @C=0.6cm{
(V \otimes I) \otimes W \ar[rr]^{\alpha_{V,I,W}}\ar[dr]_{\rho_{V} \otimes W}&& V \otimes (I \otimes W)\ar[dl]^{{V} \otimes \lambda_W}\\
& V \otimes W &
}
\end{equation}

\item A (pre-)tensor category is called \emph{strict} if all the constraints are the identity morphisms.

\item A (pre-)tensor category is called \emph{braided} if it is endowed with a \emph{braiding}, or \emph{commutativity constraint}, i.e. natural isomorphisms
$$\big(\,c_{V,W} \colon V \otimes W \overset{\sim}{\to} W \otimes V \,\big)_{V,W \in \Ob(\C)}$$ 
respecting the tensor product, in the sense of the \emph{hexagon axioms}
\begin{equation}\label{eqn:Hexagon}
\xymatrix@R=0.25cm @C=0.7cm{
&U \otimes (V \otimes W) \ar[r]^{c_{U,V \otimes W}}&(V \otimes W) \otimes U \ar[dr]^{\alpha_{V,W,U}}&\\
(U \otimes V) \otimes W \ar[dr]_-{c_{U,V} \otimes W\quad} \ar[ur]^{\alpha_{U,V,W}}&&&V \otimes (W \otimes U)\\
&(V \otimes U) \otimes W \ar[r]_{\alpha_{V,U,W}}&V \otimes (U \otimes W)\ar[ur]_{\quad V \otimes c_{U,W}}&
}
\end{equation}
\begin{equation}\label{eqn:Hexagon2}
\xymatrix@R=0.25cm @C=0.7cm{
&(U \otimes V) \otimes W \ar[r]^{c_{U \otimes V, W}}& W \otimes (U \otimes V)  \ar[dr]^{\alpha^{-1}_{W,U,V}}&\\
U \otimes (V \otimes W) \ar[dr]_-{U \otimes c_{V,W}\quad} \ar[ur]^{\alpha^{-1}_{U,V,W}}&&&(W \otimes U) \otimes V\\
&U \otimes (W \otimes V) \ar[r]_{\alpha^{-1}_{U,W,V}}& (U \otimes W) \otimes V \ar[ur]_{\quad c_{U,W} \otimes V}&
}
\end{equation}
\end{itemize}
\end{definition}

This terminology is classical except for \emph{pre-tensor categories}, which were first considered by F.~Li~\cite{PreTensor}.

In a braided monoidal category, any family of objects~$V_i$ equipped with the morphisms $\sigma_{i,j} = c_{V_i,V_j}$ form a braided system: the cYBE on $V_i \otimes V_j \otimes V_k$ follows from the naturality of~$c$ with respect to~$\sigma_{i,j}$ and~$\Id_{V_k}$, together with the hexagon axiom~\eqref{eqn:Hexagon2}. This is one of the reasons for the long-standing interest in such categories.

Let us now see how close our twisted representation categories of crossed modules of shelves are to braided monoidal categories.

\begin{theorem}[Pre-tensor categories $(\MMM^{\operatorname{tw}}_{\bullet}(R,S), \otimes)$ and $(\MMM^{\operatorname{R; tw}}_{\bullet}(R,S), \wotimes)$]\label{thm:TensorRepCrModSh}
$ $ 

\begin{enumerate}
\item Take a crossed module of shelves $(R,S,\pi,\cdot)$. The tensor product~$\otimes$ from~\eqref{eqn:TensorRepCrModSh} defines a strict pre-tensor structure on its twisted representation category $\MMM^{\operatorname{tw}}_{\bullet}(R,S)$. 
\item Take a crossed module of racks $(R,S,\pi,\cdot)$. The tensor product~$\wotimes$ from~\eqref{eqn:TensorRepCrModSh} and the maps
\begin{align*}
\alpha_{M,M',M''} \colon (M \otimes M') \otimes M'' &\overset{\sim}{\to} M \otimes (M' \otimes M''), \\
(m \otimes m') \otimes m'' &\mapsto m \mop \pi(gr(m'')) \otimes (m' \otimes m'')
\end{align*}
define a pre-tensor structure on $\MMM^{\operatorname{R; tw}}_{\bullet}(R,S)$. 
\end{enumerate}
\end{theorem}

\begin{remark}\label{rmk:TensorUntwisted}
For both pre-tensor structures, usual (i.e., non-twisted) representations form pre-tensor subcategories $\MMM^{(\operatorname{R})}_{\bullet}(R,S)$. 
\end{remark}

\begin{proof}
The verifications for the tensor product~$\otimes$ are straightforward. The tensor product~$\wotimes$ deserves more attention. We have seen that it is a bifunctor. Further, the maps $\alpha_{M,M',M''}$
\begin{itemize}
\item intertwine the $S$-actions, since
\begin{align*}
(m \mop &\pi(gr(m''))) \mop s = (m \mop s) \mop (\pi(gr(m'')) \op s) \\
&=(m \mop s) \mop \pi(gr(m'') \cdot s)=(m \mop s) \mop \pi(gr(m'' \mop s));
\end{align*}
\item intertwine the twisting maps, because of
\begin{align*}
f(m \mop \pi(gr(m''))) &= f(m)\mop \pi(gr(m'')) = f(m)\mop \pi(gr(f(m'')));
\end{align*}
\item respect the $R$-gradings:
\begin{align*}
gr(m \mop &\pi(gr(m''))) \op (gr(m') \op gr(m'')) \\
&= (gr(m) \cdot \pi(gr(m''))) \op (gr(m') \op gr(m'')) \\
&= (gr(m) \op gr(m'')) \op (gr(m') \op gr(m'')) \\
&= (gr(m) \op gr(m')) \op gr(m'');  
\end{align*}
\item are bijective, since the maps $M \to M$, $m \mapsto m \mop \pi(gr(m''))$ are so for the $S$-rack-module $(M, \mop)$.
\end{itemize}
Hence the $\alpha_{M,M',M''}$ are invertible morphisms in $\MMM^{\operatorname{tw}}_{\Set}(R,S)$.
 The naturality of~$\alpha$ follows from the fact that morphisms in $\MMM^{\operatorname{tw}}_{\Set}(R,S)$ preserve the $R$-gradings and intertwine the $S$-actions. It remains to check the pentagon axiom~\eqref{eqn:Pentagon}. Explicitly, its right-hand side sends an element $((m \otimes m') \otimes m'') \otimes m''' \in ((M \otimes M') \otimes M'')\otimes M'''$ to
$$ (m \mop \pi(gr(m''')))\mop \pi(gr_{\widetilde{\otimes}}(m'' \otimes m''')) \otimes (m'\mop \pi(gr(m''')) \otimes (m'' \otimes m''')),$$ 
while the left-hand side sends it to
$$(m \mop \pi(gr(m'')))\mop \pi(gr(m''')) \otimes (m'\mop \pi(gr(m''')) \otimes (m'' \otimes m''')).$$ 
Recalling that~$\pi$ is a shelf morphism, one obtains
\begin{align*}
\pi(gr_{\widetilde{\otimes}}(m'' \otimes m''')) &= \pi(gr(m'') \op gr(m''')) = \pi(gr(m'')) \op \pi(gr(m''')),
\end{align*}
and the defining property of a rack action for~$\mop$ yields
\begin{align*}
(m \mop \pi(gr(m''')))\mop (\pi(gr(m'')) \op& \pi(gr(m'''))) = \\
&(m \mop \pi(gr(m'')))\mop \pi(gr(m''')),
\end{align*}
hence our pentagon axiom is satisfied.
\end{proof}

\begin{remark}\label{rmk:TensorShelves}
For the twisted representation category $\MMM^{\operatorname{tw}}_{\bullet}(R,S)$ of a crossed module of \emph{shelves}, the tensor product~$\wotimes$ and the $\alpha_{M,M',M''}$ above satisfy all the pre-tensor structure axioms except for the invertibility of $\alpha$.
\end{remark}

Let us next study the unitality of our categories. In order to admit an isomorphism $I \otimes V \overset{\sim}{\to} V$ or $V \otimes I \overset{\sim}{\to} V$, the unit~$I$ has to be a one-element set with a twisted representation structure over $(R,S,\pi,\cdot)$, or, in other words, a Yetter-Drinfel$'$d character (Definition~\ref{def:char_YD}) for the corresponding rank~$2$ braided system (Proposition~\ref{prop:CrModShBrSyst}). The $S$-action and the twisting have to be the unique maps $S \to I$ and $I\to I$ respectively. Further, the single element of~$I$ should be graded by an \emph{$S$-invariant} element $r_0 \in R$, in the sense of $r_0 \cdot s = r_0$ for all $s \in S$. Summarizing, one gets

\begin{proposition}[YD characters in $\YD_S^R$]\label{prop:YDCharSh} 
$ $

For a crossed module of shelves $(R,S,\pi,\cdot)$, a complete list of Yetter-Drinfel$'$d characters in $\YD_S^R$ (up to isomorphism) is indexed by $S$-invariant elements $r_0 \in R$, and given by the structures
$$I_{r_0} \,= \,\big(\, \{\ast\},\, \rho_{r_0} \colon s \mapsto \ast,\, \delta_{r_0} \colon \ast \mapsto r_0 \, \big).$$
\end{proposition}

\begin{example}
A shelf $(S,\op)$ is called \emph{pointed} if it contains a preferred element~$e$ satisfying $e \op s = e$, $s \op e = s$ for all $s \in S$. A conjugation rack yields a classical example, with the neutral element of the underlying group chosen as~$e$. The crossed module of shelves $(S,S,\Id_S,\op)$ associated to a pointed shelf (Example~\ref{ex:AdjRep}) comes with the YD character~$I_{e}$ in $\YD_S^S$.
\end{example}

A YD character~$I_{r_0}$ can be seen as a left unit for the tensor structure~$\otimes$ on $\MMM^{\operatorname{tw}}_{\bullet}(R,S)$, since the maps $\lambda_{M} \colon I_{r_0} \otimes M \overset{\sim}{\to} M$, $\ast \otimes m \mapsto m$ define a natural isomorphism. On the other hand, $I_{r_0}$ can be seen as a right unit for the tensor structure~$\wotimes$, since the maps $\rho_{M} \colon M \wotimes I_{r_0} \overset{\sim}{\to} M$, $m \wotimes \ast \mapsto m \cdot \pi(r_0)$ define a natural morphism, which becomes an isomorphism in the case of rack-modules. Unfortunately, the authors do not know how to complete at least one of these structures into a whole unit constraint.

Finally, recall the braidings $\sigma_{TwCrModSh}$ for objects in $\MMM^{\operatorname{tw}}_{\bullet}(R,S)$ (Remark~\ref{rmk:BrYDSR}). They are natural candidates for forming a commutativity constraint for $(\MMM^{\operatorname{tw}}_{\bullet}(R,S), \otimes)$ or $(\MMM^{\operatorname{R; tw}}_{\bullet}(R,S), \wotimes)$. However, these maps do not respect the $R$-gradings, so they are not even morphisms in the corresponding categories! On the other hand, they intertwine the $S$-actions and the twistings, form a natural family, and in the rack case admit inverses.

The representation category of a crossed module of Leibniz algebras is also pre-tensor with an interesting associator, as we now establish.

\begin{proposition}[Representations of a crossed module of Leibniz algebras as enriching structures]\label{RepCrModLAAsEnrich}
$ $

Take a crossed module of Leibniz algebras $(\k, \g, \pi, \cdot)$, and consider the braided system $(\kp, \gp; \osigma)$ from Proposition~\ref{prop:CrModLABrSyst}. Then one has a functor
\begin{align*}
Z_{CrModLA} \colon \YD_{\gp}^{\kp} &\to \ZZZ_{\gp}^{\kp},\\
(M,\ast,\delta) &\mapsto (M,\sigma_{\kp,M},\sigma_{M,\gp}),\\
\text{where} \qquad &\sigma_{\kp,M} (k \otimes m) = m_{(0)} \otimes k \cdot m_{(1)},\\
&\sigma_{M,\gp}(m \otimes g) = g_{(1)} \otimes m\ast g_{(2)},
\end{align*}
using the usual Sweedler's notations for the $\kp$-coaction $\delta$ on~$M$ and for the comultiplication~$\Delta$ on~$\gp$, as well as the unitarized adjoint action~$\cdot$ of~$\kp$ on itself (Lemma~\ref{l:Unit}).
\end{proposition}

The proof is similar to that of Proposition~\ref{YDModAsEnrich}. Note that, in contrast to the situation there, one has no hope of having a category inclusion here, since the map~$\sigma_{\kp,M}$ is not sufficient for reconstructing the coaction~$\delta$ in general.  

As usual, combining~$Z_{CrModLA}$ with the functor~$E_{\gp}^{\kp}$ (Proposition~\ref{prop:center}), one obtains a bifunctor $\otimes = E_{\gp}^{\kp} \circ(\Id \times Z_{CrModLA})$ on~$\YD_{\gp}^{\kp}$, restricting to a bifunctor on the representation category $\MMM(\k,\g)$. The corresponding product structure is explicitly written as follows:

\begin{proposition}[Product of representations of a crossed module of Leibniz algebras]\label{ProdRepsCrModLA}
$\quad$

Given a crossed module of Leibniz algebras $(\k, \g, \pi, \cdot)$ and two YD modules $(M, \ast, \delta)$ and $(M', \ast', \delta')$ over the braided system $(\kp, \gp; \osigma)$ (Proposition~\ref{prop:CrModLABrSyst}), their product $M \otimes M'$ can be endowed with the following YD module structure:
\begin{itemize}
\item the unit $1 \in \gp$ acts on $M \otimes M'$ by the identity, and elements $g \in \g$ according to the Leibniz rule: 
$$(m \otimes m') \ast_{\otimes} g = m \otimes m'\ast' g  + m\ast g  \otimes m';$$
\item the $\kp$-coaction is given by
$$\delta_{\otimes}(m \otimes m') = m_{(0)}\otimes m'_{(0)} \otimes m_{(1)} \cdot m'_{(1)},$$
with the same notations as in Proposition~\ref{RepCrModLAAsEnrich}.
\end{itemize}
\end{proposition}

\begin{theorem}[Pre-tensor structure on~$\MMM(\k,\g)$]\label{thm:TensorRepCrModLA}
$ $ 

Take a crossed module of Leibniz algebras $(\k, \g, \pi, \cdot)$. Consider the tensor product~$\otimes$ from Proposition~\ref{ProdRepsCrModLA} and, for $M,M',M'' \in \YD_{\gp}^{\kp}$, the maps
\begin{align*}
\alpha_{M,M',M''} \colon (M \otimes M') \otimes M'' &\to M \otimes (M' \otimes M''), \\
(m \otimes m') \otimes m'' &\mapsto m \ast \pi(m''_{(1)}) \otimes (m' \otimes m''_{(0)}).
\end{align*}
\begin{enumerate}
\item These data satisfy all the pre-tensor structure axioms except for the invertibility of~$\alpha$.
\item The YD module $(\kk, \varepsilon, \nu)$, with $\nu(1) = 1 \in \kp$, is a strict right unit for this structure.
\item Restricted to the representation category $\MMM(\k,\g)$, this yields a genuine pre-tensor structure with a right unit.
\end{enumerate}
\end{theorem}

\begin{proof}
\begin{enumerate}
\item We first show that $\alpha_{M,M',M''}$ is a morphism in~$\YD_{\gp}^{\kp}$. To show that it intertwines $\gp$-actions, one needs to check that
\begin{align*}
& (m\ast g_{(1)}) \ast \pi((m''\ast g_{(3)})_{(1)}) \otimes (m'\ast g_{(2)} \otimes (m''\ast g_{(3)})_{(0)}) =\\
& (m \ast \pi(m''_{(1)}))\ast g_{(1)} \otimes (m'\ast g_{(2)} \otimes m''_{(0)}\ast g_{(3)})
\end{align*}
for all $g \in \gp$. Using the cocommutativity of the comultiplication on~$\gp$ and the compatibility relation~\eqref{eqn:YDforCrModLA'} between~$\delta$ and~$\ast$, the first expression rewrites as
$$(m\ast g_{(1)}) \ast \pi(m''_{(1)}\cdot g_{(2)}) \otimes (m'\ast g_{(3)} \otimes m''_{(0)}\ast g_{(4)}).$$
Due to~\eqref{eqn:CrModLAUn'} and~\eqref{eqn:YDforCrModLA}, one has
\begin{align*}
(m\ast g_{(1)}) \ast \pi(m''_{(1)}\cdot g_{(2)}) & = (m\ast g_{(1)}) \ast (\pi(m''_{(1)})\cdot g_{(2)})\\
& = (m \ast \pi(m''_{(1)})\ast g_{(1)},
\end{align*}
so the desired expressions coincide.

We next verify that $\alpha_{M,M',M''}$ intertwines $\kp$-coactions. One calculates
\begin{align*}
\delta(m \ast \pi(m''_{(1)})) &\overset{\eqref{eqn:YDforCrModLA'}}{=}m_{(0)} \ast (\pi(m''_{(1)}))_{(1)} \otimes m_{(1)} \cdot (\pi(m''_{(1)}))_{(2)}\\
&\overset{\eqref{eqn:CrModLAUn''}}{=}m_{(0)} \ast \pi((m''_{(1)})_{(1)}) \otimes m_{(1)} \cdot \pi((m''_{(1)})_{(2)})\\
&\overset{\eqref{eqn:CrModLAUn}}{=}m_{(0)} \ast \pi((m''_{(1)})_{(1)}) \otimes m_{(1)} \cdot (m''_{(1)})_{(2)},
\end{align*}
The desired intertwining relation then rewrites as
\begin{align*}
& (m \ast \pi(m''_{(1)}) \otimes (m' \otimes m''_{(0)})) \otimes (m_{(1)} \cdot m'_{(1)}) \cdot m'_{(2)}= \\
& (m_{(0)} \ast \pi((m''_{(2)})_{(1)}) \otimes (m'_{(0)} \otimes m''_{(0)})) \otimes (m_{(1)} \cdot (m''_{(2)})_{(2)}) \cdot (m'_{(1)} \cdot m''_{(1)}).
\end{align*}
Using relation~\ref{eqn:YDforCrModLA'''} and the cocommutativity of the comultiplication on~$\kp$, the latter expression equals
\begin{align*}
& (m_{(0)} \ast \pi((m''_{(1)})_{(1)}) \otimes (m'_{(0)} \otimes f(m''_{(0)}))) \otimes (m_{(1)} \cdot (m''_{(1)})_{(2)}) \cdot (m'_{(1)} \cdot (m''_{(1)})_{(3)}),
\end{align*}
which using~\eqref{eqn:YDforCrModLA} for the adjoint action on~$\kp$ becomes
\begin{align*}
& (m_{(0)} \ast \pi((m''_{(1)})_{(1)}) \otimes (m'_{(0)} \otimes f(m''_{(0)}))) \otimes (m_{(1)} \cdot m'_{(1)}) \cdot (m''_{(1)})_{(2)}).
\end{align*}
One more application of~\ref{eqn:YDforCrModLA'''} transforms our expression into the desired form. 

The naturality of~$\alpha$ is straightforward. It remains to verify the pentagon axiom, which here reads
\begin{align*}
  &(m \ast \pi(m''_{(1)})) \ast \pi(m'''_{(2)}) \otimes
 (m'\ast \pi(m'''_{(1)}) \otimes (m''_{(0)} \otimes m'''_{(0)}))= \\
  (m \ast &\pi((m'''_{(2)})_{(1)})) \ast \pi(m''_{(1)} \cdot m'''_{(1)}) \otimes
 (m'\ast \pi((m'''_{(2)})_{(2)}) \otimes (m''_{(0)} \otimes m'''_{(0)})).
\end{align*}
It is done by an argument similar to those used for intertwining properties, juggling relations from Lemmas~\ref{l:Unit} and~\ref{l:YDforCrModLA}.

\item Straightforward verifications.

\item One easily checks that the inverse of~$\alpha_{M,M',M''}$ is given by the map
\begin{align*}
m \otimes (m' \otimes m'') &\mapsto (m \ast S(\pi(m''_{(1)})) \otimes m') \otimes m''_{(0)},
\end{align*}
where $S \colon \gp \to \gp$ is the ``antipode-like'' map defined by $S(1)=1$ and $S(g) = -g$ for $g \in \g$. \qedhere
\end{enumerate} 
\end{proof}

\begin{remark}
The associators from Theorems~\ref{thm:TensorRepCrModSh} and~\ref{thm:TensorRepCrModLA} can be written in a uniform way as 
$(m \otimes m') \otimes m'' \mapsto \underline{m}  \otimes (m' \otimes \underline{m''}),$ 
using the formal notation $\sigma_{genYD}(m \otimes m'') = \underline{m''} \otimes \underline{m}$. The resemblance between the two pre-tensor structures is more than a simple coincidence: crossed modules of both shelves and Leibniz algebras can be unified in the framework of \emph{categorical shelves}, developed in~\cite{CatSelfDistr,Lebed3,RackBig}.
\end{remark}

As in the case of representations of crossed modules of shelves, the braidings~$\sigma_{CrModLA}$ from Theorem~\ref{thm:SrModLABr} are not $\kp$-comodule maps in general, and thus do not provide a braided structure on our pre-tensor category $\MMM(\k,\g)$.

Summing up, we have constructed several new pre-tensor categories with global braidings (in the Yang-Baxter sense) which do not stem from a braiding structure on the category. It would be interesting to determine whether our braidings can be rendered categorical for a different (pre-)tensor structure on $\MMM^{\operatorname{(R;) tw}}_{\bullet}(R,S)$ or $\MMM(\k,\g)$, or there is a conceptual reason preventing such a structure to exist.

\footnotesize
\bibliographystyle{abbrv}
\bibliography{biblio}

\begin{thebibliography}{10}

\bibitem{RackBig}
C.~{Alexandre}, M.~{Bordemann}, S.~{Riviere}, and F.~{Wagemann}.
\newblock {Rack-bialgebras and star-products on duals of {L}eibniz algebras}.
\newblock {\em ArXiv e-prints}, Dec. 2014.

\bibitem{Bantay}
P.~Bantay.
\newblock Characters of crossed modules and premodular categories.
\newblock In {\em Moonshine: the first quarter century and beyond}, volume 372
  of {\em London Math. Soc. Lecture Note Ser.}, pages 1--11. Cambridge Univ.
  Press, Cambridge, 2010.

\bibitem{Beck}
J.~Beck.
\newblock Distributive laws.
\newblock In {\em Sem. on {T}riples and {C}ategorical {H}omology {T}heory
  ({ETH}, {Z}\"urich, 1966/67)}, pages 119--140. Springer, Berlin, 1969.

\bibitem{EntwAlgCoalg}
T.~Brzezi{\'n}ski and S.~Majid.
\newblock Coalgebra bundles.
\newblock {\em Comm. Math. Phys.}, 191(2):467--492, 1998.

\bibitem{CaMiZhu}
S.~Caenepeel, G.~Militaru, and S.~Zhu.
\newblock Crossed modules and {D}oi-{H}opf modules.
\newblock {\em Israel J. Math.}, 100:221--247, 1997.

\bibitem{CatSelfDistr}
J.~S. Carter, A.~S. Crans, M.~Elhamdadi, and M.~Saito.
\newblock Cohomology of categorical self-distributivity.
\newblock {\em J. Homotopy Relat. Struct.}, 3(1):13--63, 2008.

\bibitem{CransWag}
A.~Crans and F.~Wagemann.
\newblock Crossed modules of racks.
\newblock {\em Homology, Homotopy and its Appl.}, 16(2):85--106, 2014.

\bibitem{Crans}
A.~S. Crans.
\newblock {\em Lie 2-algebras}.
\newblock ProQuest LLC, Ann Arbor, MI, 2004.
\newblock Thesis (Ph.D.)--University of California, Riverside.

\bibitem{Cuvier}
C.~Cuvier.
\newblock Homologie de {L}eibniz et homologie de {H}ochschild.
\newblock {\em C. R. Acad. Sci. Paris S\'er. I Math.}, 313(9):569--572, 1991.

\bibitem{FRT2}
L.~Faddeev, N.~Reshetikhin, and L.~Takhtajan.
\newblock Quantum groups.
\newblock In {\em Braid group, knot theory and statistical mechanics}, volume~9
  of {\em Adv. Ser. Math. Phys.}, pages 97--110. World Sci. Publ., Teaneck, NJ,
  1989.

\bibitem{FRT1}
L.~D. Faddeev, N.~Y. Reshetikhin, and L.~A. Takhtajan.
\newblock Quantization of {L}ie groups and {L}ie algebras.
\newblock In {\em Algebraic analysis, {V}ol.\ {I}}, pages 129--139. Academic
  Press, Boston, MA, 1988.

\bibitem{FoxMarkl}
T.~F. Fox and M.~Markl.
\newblock Distributive laws, bialgebras, and cohomology.
\newblock In {\em Operads: {P}roceedings of {R}enaissance {C}onferences
  ({H}artford, {CT}/{L}uminy, 1995)}, volume 202 of {\em Contemp. Math.}, pages
  167--205. Amer. Math. Soc., Providence, RI, 1997.

\bibitem{Hennings}
M.~A. Hennings.
\newblock On solutions to the braid equation identified by {W}oronowicz.
\newblock {\em Lett. Math. Phys.}, 27(1):13--17, 1993.

\bibitem{Joyce}
D.~Joyce.
\newblock A classifying invariant of knots, the knot quandle.
\newblock {\em J. Pure Appl. Algebra}, 23(1):37--65, 1982.

\bibitem{KraSle}
U.~{Kraehmer} and P.~{Slevin}.
\newblock {Factorisations of distributive laws}.
\newblock {\em ArXiv e-prints}, Sept. 2014.

\bibitem{Lebed}
V.~Lebed.
\newblock {\em Braided objects: unifying algebraic structures and categorifying
  virtual braids}.
\newblock 2012.
\newblock Thesis (Ph.D.)--Universit{\'e} Paris 7.

\bibitem{Lebed2}
V.~Lebed.
\newblock Braided systems: a unified treatment of algebraic structures with
  several operations.
\newblock {\em ArXiv e-prints}, May 2013.

\bibitem{Lebed3}
V.~Lebed.
\newblock Categorical aspects of virtuality and self-distributivity.
\newblock {\em J. Knot Theory Ramifications}, 22(9):1350045, 32, 2013.

\bibitem{Lebed1}
V.~Lebed.
\newblock Homologies of algebraic structures via braidings and quantum
  shuffles.
\newblock {\em J. Algebra}, 391:152--192, 2013.

\bibitem{Lebed2ter}
V.~Lebed.
\newblock R-matrices, {Y}etter-{D}rinfel'd modules and {Y}ang-{B}axter
  equation.
\newblock {\em Axioms}, 2(3):443--476, 2013.

\bibitem{PreTensor}
F.~Li.
\newblock The right braids, quasi-braided pre-tensor categories, and general
  {Y}ang-{B}axter operators.
\newblock {\em Comm. Algebra}, 32(2):397--441, 2004.

\bibitem{BialgBimonads}
M.~Livernet, B.~Mesablishvili, and R.~Wisbauer.
\newblock Generalised bialgebras and entwined monads and comonads.
\newblock {\em J. Pure Appl. Algebra}, 219(8):3263--3278, 2015.

\bibitem{Cyclic}
J.-L. Loday.
\newblock {\em Cyclic homology}, volume 301 of {\em Grundlehren der
  Mathematischen Wissenschaften [Fundamental Principles of Mathematical
  Sciences]}.
\newblock Springer-Verlag, Berlin, 1992.
\newblock Appendix E by Mar{\'{\i}}a O. Ronco.

\bibitem{LoLei}
J.-L. Loday.
\newblock Une version non commutative des alg\`ebres de {L}ie: les alg\`ebres
  de {L}eibniz.
\newblock {\em Enseign. Math. (2)}, 39(3-4):269--293, 1993.

\bibitem{GenBialg}
J.-L. Loday.
\newblock Generalized bialgebras and triples of operads.
\newblock {\em Ast\'erisque}, (320):x+116, 2008.

\bibitem{LoPi}
J.-L. Loday and T.~Pirashvili.
\newblock Universal enveloping algebras of {L}eibniz algebras and (co)homology.
\newblock {\em Math. Ann.}, 296(1):139--158, 1993.

\bibitem{MaiSch}
J.~Maier and C.~Schweigert.
\newblock Modular categories from finite crossed modules.
\newblock {\em J. Pure Appl. Algebra}, 215(9):2196--2208, 2011.

\bibitem{Matveev}
S.~V. Matveev.
\newblock Distributive groupoids in knot theory.
\newblock {\em Mat. Sb. (N.S.)}, 119(161)(1):78--88, 160, 1982.

\bibitem{Bimonads}
B.~Mesablishvili and R.~Wisbauer.
\newblock Bimonads and {H}opf monads on categories.
\newblock {\em J. K-Theory}, 7(2):349--388, 2011.

\bibitem{YBE}
J.~H.~H. {Perk} and H.~{Au-Yang}.
\newblock {Yang-Baxter equations}.
\newblock In {\em {Encyclopedia of Mathematical Physics}}, volume~5, pages
  465--473. Elsevier Science, Oxford, 2006.

\bibitem{Rad2}
D.~E. Radford.
\newblock Solutions to the quantum {Y}ang-{B}axter equation and the
  {D}rinfel$'$d double.
\newblock {\em J. Algebra}, 161(1):20--32, 1993.

\bibitem{Wor}
S.~L. Woronowicz.
\newblock Solutions of the braid equation related to a {H}opf algebra.
\newblock {\em Lett. Math. Phys.}, 23(2):143--145, 1991.

\bibitem{Yetter}
D.~N. Yetter.
\newblock Quantum groups and representations of monoidal categories.
\newblock {\em Math. Proc. Cambridge Philos. Soc.}, 108(2):261--290, 1990.

\end{thebibliography}
\end{document}